%% file: main.tex
\documentclass[12pt]{article}

\usepackage{preprint-layout}
\usepackage{algorithm}
\usepackage{algpseudocodex}
\usepackage[notodo]{preprint-tools}
\usepackage{graphicx}
\graphicspath{{figures/}}

\title{Tangent-Space Multiscale Manifold Methods for Nonlinear Elliptic Problems}
\author{Mats G. Larson \mbox{ } Anna Persson}
\date{}

\begin{document}

\maketitle

\begin{abstract}
We introduce a tangent-space multiscale manifold method for nonlinear heterogeneous elliptic problems. The method represents the fine-scale solution by a nonlinear reconstruction of a coarse state. The ideal reconstruction eliminates fine scales through a constrained variational problem, and the computable reconstruction approximates this map by localized nonlinear patch solves blended with a partition of unity. Because the approximation set is a nonlinear manifold, the coarse equation is posed with tangent multiscale test functions. We also formulate a network-interpolated variant in which only the restricted patch outputs used by the partition-of-unity blend, together with their tangent actions, are approximated by local learned maps. For heterogeneous monotone nonlinear diffusion, we record the structural monotonicity, differentiability, patch-map regularity, and conditional perturbation estimates that separate the geometric stability mechanism from localization, residual, and optional learning defects. A rigorous a priori theory for the decay of the localization defect, and the resulting convergence rates in the coarse mesh size, is deferred to a separate analysis; here these defects are controlled conditionally and their decay is demonstrated numerically.
\end{abstract}


\input{sections/introduction_v7}
\input{sections/global_fine_scale_manifold_v4}
\input{sections/localized_reconstruction_v4}
\input{sections/network_interpolation_v4}
\input{sections/computable_coarse_problems_v4}
\input{sections/heterogeneous_nonlinear_diffusion_v3}
\input{sections/error_analysis_v5}
\input{sections/numerical_experiments_v0}
\appendix
\input{sections/appendix_extended_algorithm_v0}
\input{sections/acknowledgements_v1}
\input{sections/references_v4}

\input{sections/author_addresses_v1}

\end{document}

%% file: sections/introduction_v7.tex

\section{Introduction}
\label{sec:introduction}

\paragraph{Background.}
\label{par:introduction-background}
Multiscale elliptic problems with rapidly varying coefficients arise in porous media, composite materials, heterogeneous diffusion, and nonlinear continuum models. Direct resolution on the finest scale is often too expensive, while standard coarse finite element spaces do not contain the oscillatory response induced by the coefficient field. We focus on applications with general rough coefficient fields, without assumed periodicity, self-similarity, or scale separation. In this setting, classical analytical homogenization, which derives an effective equation from precisely such structural assumptions, does not apply. The method developed in this paper instead belongs to a class of numerical homogenization methods that construct problem-adapted coarse approximations computationally from the actual coefficient field and require no such structural assumptions. For nonlinear constitutive laws the fine-scale response is also state dependent. A useful coarse method must therefore retain the variational structure of the fine-scale problem, localize the fine-scale work, and keep the online unknowns on a low-dimensional coarse space; this coarse--fine viewpoint is closely related to variational multiscale methods, bubble-function enrichments, and residual-based stabilized finite element methods \cite{Hughes1995,HughesFeijooMazzeiQuincy1998,HughesSangalli2007,BrezziBristeauFrancaMalletRoge1992,BaiocchiBrezziFranca1993,BarredaMadureira2019}.

\paragraph{Localized Fine-Scale Elimination.}
\label{par:introduction-localized-fine-scale-elimination}
Classical multiscale finite element and heterogeneous multiscale methods incorporate local fine-scale information through adapted basis functions, cell problems, or effective coefficients \cite{HouWu1997,EfendievHou2009,Abdulle2009}. A complementary principle is to eliminate unresolved scales by a variational splitting of the fine finite element space and to approximate the fine-scale part by local problems. This localization idea appears already in the adaptive variational multiscale method of Larson and M{\aa}lqvist \cite{LarsonMalqvist2007}, where the fine-scale part is approximated by decoupled local problems in a slice space and controlled by a posteriori estimates. For linear elliptic problems, the subsequent work \cite{MP14} gave the localized corrector construction and a priori localization theory that led to the localized orthogonal decomposition framework; see, for example, \cite{AHP21,MaP21book}. Efficient implementations and related localized fine-scale elimination techniques are now well established for linear heterogeneous problems \cite{ENGWER2019123}. Related numerical-homogenization frameworks localize the fine-scale response in other ways: the constraint energy minimizing generalized multiscale finite element method builds exponentially decaying multiscale basis functions by local energy minimization \cite{CEL18}, while operator-adapted wavelets and the associated gamblet decomposition recover the fine-scale solution through an optimal-recovery, game-theoretic formulation \cite{Owhadi2017,OwhadiScovel2019}.

\paragraph{Nonlinear State Dependence.}
\label{par:introduction-nonlinear-state-dependence}
Nonlinear problems change the structure of fine-scale elimination: the unresolved correction is no longer a linear operator applied to a coarse basis function or residual, but depends on the current coarse state. Several multiscale frameworks have been extended to nonlinear problems. Multiscale finite element methods were developed for nonlinear elliptic problems in \cite{EGH04}, and the generalized multiscale finite element method for nonlinear elliptic equations in \cite{EGLP14}. Within the heterogeneous multiscale method, quasilinear elliptic homogenization problems were analyzed in \cite{AV14} and solved by an offline--online strategy in \cite{ABV14}. Within the localized orthogonal decomposition framework, corrector-based constructions have been given for semilinear elliptic \cite{HMP14_b} and parabolic \cite{MP18} problems, and for genuinely nonlinear diffusion, through linearization, in \cite{V21,KV25}. These approaches treat the state dependence in different ways. The nonlinear multiscale finite element method solves nonlinear local problems elementwise, as nonlinear analogues of basis maps with boundary data prescribed by the coarse function, tested against standard coarse functions, with a convergence theory set in the classical homogenization framework. The heterogeneous multiscale method solves linear cell problems with the macroscopic state frozen at quadrature points and relies on scale separation for the micro--macro coupling. The generalized multiscale finite element method employs state-parametrized linear spectral spaces computed offline. The nonlinear localized orthogonal decomposition methods compute linear correctors from a linearization of the equation in each iteration. The method developed here instead represents the complete eliminated fine scale by a single nonlinear map from coarse states to fine-scale corrections and solves the resulting coarse problem on the associated nonlinear manifold. As in the localized orthogonal decomposition framework, it requires neither periodicity nor scale separation of the coefficient.

\paragraph{Manifold Viewpoint.}
\label{par:introduction-manifold-viewpoint}
Let \(V_h\) be the fine finite element space, let \(V_H\subset V_h\) be a coarse finite element space, let \(\Pi_H:V_h\to V_H\) be a projection, and set \(V_f=\operatorname{ker}\Pi_H\). For a coarse state \(v_H\in V_H\), the ideal fine-scale correction \(v_f(v_H)\in V_f\) is defined by the constrained residual equation,
\begin{equation}
A(v_H+v_f(v_H);w_f)=L(w_f) \quad \forall w_f\in V_f \label{eq:introduction-global-correction}
\end{equation}
and the corresponding reconstruction map and manifold are defined by,
\begin{equation}
M_h(v_H)=v_H+v_f(v_H),\qquad \mathcal M_h=M_h(V_H)\subset V_h \label{eq:introduction-global-manifold}
\end{equation}
The fine finite element solution is represented exactly by this manifold: if \(u_h\) solves the fine problem, then \(u_h=M_h(\Pi_Hu_h)\). The approximation enters only when the global map \(M_h\) is replaced by a localized computable map.

\paragraph{Localized Complete Fine-Scale Reconstruction.}
\label{par:introduction-localized-reconstruction}
The practical method approximates \(M_h\) by localized patch maps. The use of local patches and a partition-of-unity split is reminiscent of the partition-of-unity finite element method (PUFEM) and generalized finite element method (GFEM), where local approximation spaces are combined into a global trial space \cite{MelenkBabuska1996,BabuskaMelenk1997}. Here, however, the split is used primarily to localize the fine-scale equations for the eliminated correction, rather than to prescribe an enriched global trial space; this viewpoint is also related to multiscale GFEM constructions based on local approximation spaces \cite{BabuskaLipton2011}. For each partition-of-unity function \(\varphi_{H,i}\), we solve a nonlinear fine-scale problem on an oversampled patch \(\omega_i^{\ell}\). The local input is the vector \(P_i v_H\) of coarse degrees of freedom used on that patch. Since the global blend only sees the support of \(\varphi_{H,i}\), only the restricted output \(R_i v_{f,i}^{\ell}(P_i v_H)\) on \(\omega_i^0=\operatorname{int}(\operatorname{supp}\varphi_{H,i})\) is needed. The localized reconstruction is defined by,
\begin{align}
M_h^{\ell}(v_H)
&=v_H+(I-\Pi_H)I_h\Bigl(
  \sum_{i\in\mathcal I}\varphi_{H,i}E_{i,0}^{\operatorname{ext}}
  R_i v_{f,i}^{\ell}(P_i v_H)
  \Bigr) \label{eq:introduction-localized-reconstruction}
\end{align}
where \(I_h\) is fine nodal interpolation and \(E_{i,0}^{\operatorname{ext}}\) is extension by zero from \(\omega_i^0\) to the global fine grid. The final projection \(I-\Pi_H\) returns the blended correction to the fine-scale space. Thus the patch solutions are local representatives of the complete nonlinear fine-scale correction, not nonlinear analogues of individual basis correctors.

\paragraph{Tangent Coarse Equation.}
\label{par:introduction-tangent-coarse-equation}
Because \(M_h^{\ell}(V_H)\) is a nonlinear manifold, the coarse residual must be tested with tangent multiscale functions. The patch-solved method seeks \(u_H^{\ell}\in V_H\) such that,
\begin{align}
A\bigl(M_h^{\ell}(u_H^{\ell});D M_h^{\ell}(u_H^{\ell})[\delta v_H]\bigr)
&=L\bigl(D M_h^{\ell}(u_H^{\ell})[\delta v_H]\bigr)
  \quad \forall \delta v_H\in V_H \label{eq:introduction-tangent-coarse-problem}
\end{align}
and then sets \(u_{\mathrm{ms}}^{\ell}=M_h^{\ell}(u_H^{\ell})\). The same tangent formulation applies at the ideal level with \(M_h\) in place of \(M_h^{\ell}\). The tangent test space is not an implementation detail; it is the natural variational test space on the nonlinear approximation manifold.

\paragraph{Local Interpolation Layer.}
\label{par:introduction-local-interpolation-layer}
The localized patch maps are finite-dimensional nonlinear solution operators. They may be evaluated by direct patch solves, or approximated offline and reused online. We suggest an optional network-interpolated method that replaces only the restricted local output maps and their tangent actions by learned maps, giving a reconstruction \(M_{h,\theta}^{\ell}\). This differs from global operator-learning approaches, which approximate maps between full function spaces \cite{LuJinKarniadakis2021,KovachkiLanthalerMishra2023}, from model-reduction strategies in which neural networks parametrize a global low-dimensional solution manifold \cite{BHKS21}, and from classical reduced-basis methodology, which builds a global reduced space for a full parameter family \cite{RozzaHuynhPatera2008,HesthavenRozzaStamm2016}. It is closer to learned localized numerical-homogenization surrogates \cite{KMP22,KMP23,KPU25}, but here the learned quantities are the support-restricted patch outputs that enter the partition-of-unity reconstruction. The deterministic localized method remains the primary object; learning is an optional interpolation layer.

\paragraph{Contributions.}
\label{par:introduction-contributions}
The contributions are as follows. First, we formulate an ideal nonlinear fine-scale elimination in which the coarse variable parametrizes the fine finite element solution through the manifold map \(M_h\). Second, we derive the tangent-space coarse equation associated with this nonlinear manifold. Third, we introduce the localized reconstruction \(M_h^{\ell}\), including the support-restricted patch outputs used in the partition-of-unity blend. Fourth, we give patch-solved and network-interpolated computable coarse problems in a common residual form. Fifth, for heterogeneous monotone nonlinear diffusion, we verify monotonicity, local Lipschitz continuity, tangent coercivity, finite-dimensional patch-map regularity, and establish a conditional perturbation estimate in which the chord curvature is absorbed into the stability constant while localization and reduced-residual defects remain on the right-hand side.

We emphasize that the present paper is deliberately methodological. In the error analysis, the localization defect \(\eta_{\mathrm{loc}}(H,\ell)\) and the learning defect enter only as quantities on the right-hand side, and the estimates are conditional on their smallness. A rigorous a priori theory for these defects---in particular the exponential decay of the localization error in the oversampling radius \(\ell\), and the resulting convergence rates in the coarse mesh size \(H\)---is \emph{not} established here; it is the subject of a separate, more theoretical companion manuscript. In the present paper the decay of the localization defect is instead demonstrated numerically. Accordingly, the contribution of this paper is the tangent-space manifold formulation, the computable localized and network-interpolated reconstructions, and the conditional perturbation mechanism that transfers reconstruction and residual estimates to the final multiscale error, rather than the localization theory itself.

\paragraph{Outline.}
\label{par:introduction-outline}
Section \ref{sec:global-fine-scale-manifold} introduces the ideal nonlinear reconstruction map, the fine-scale manifold, and the tangent-space coarse equation. Section \ref{sec:localized-reconstruction} defines the localized patch construction and the partition-of-unity reconstruction. Section \ref{sec:network-interpolation} formulates the network interpolation of restricted local patch outputs and tangent actions. Section \ref{sec:computable-coarse-problems} states the patch-solved and network-interpolated coarse problems in a common algebraic form. Section \ref{sec:heterogeneous-nonlinear-diffusion} verifies the structural assumptions for heterogeneous monotone nonlinear diffusion. Section \ref{sec:error-analysis} gives the geometric error split, the chord-curvature kickback, and the conditional localization and residual estimates that organize the error analysis. Section \ref{sec:numerical-experiments} presents numerical experiments for heterogeneous monotone nonlinear diffusion, covering the localization and coarse-mesh convergence of the patch-solved method, the network-interpolated solve, and a backward-Euler parabolic test.

%% file: sections/global_fine_scale_manifold_v4.tex

\section{Global Fine-Scale Manifold}
\label{sec:global-fine-scale-manifold}

This section defines the ideal nonlinear reconstruction map. It is global and therefore not intended for computation, but it identifies the fine-scale object that the localized method approximates. In the error analysis we later specialize to monotone equations; for the formulation of the method we keep only the assumptions needed to define the maps and tangent spaces below.

\paragraph{Meshes and Finite Element Spaces.}
\label{par:meshes-and-finite-element-spaces}
Let \(\Omega\) be a bounded polygonal or polyhedral domain. We are given two nested, shape-regular and quasi-uniform triangulations of \(\Omega\): a coarse triangulation \(\mathcal T_H\) with mesh-width parameter \(H\coloneqq\max_{T\in\mathcal T_H}\operatorname{diam}(T)\), and a fine triangulation \(\mathcal T_h\) with mesh-width parameter \(h\coloneqq\max_{K\in\mathcal T_h}\operatorname{diam}(K)\), where \(\mathcal T_h\) is a refinement of \(\mathcal T_H\) so that every coarse element is a union of fine elements and \(h\le H\). On these triangulations we use continuous piecewise-linear \(P_1\) Lagrange finite elements with homogeneous Dirichlet boundary conditions, and denote the resulting finite element spaces by \(V_H\) and \(V_h\). The nestedness of the triangulations yields the conforming inclusion \(V_H\subset V_h\) used throughout.

\paragraph{Fine-Scale Reference Problem.}
\label{par:fine-scale-reference-problem}
Let \(\|\cdot\|_V\) denote the energy norm used in the analysis. The reference problem is to find \(u_h\in V_h\) such that
\begin{equation}
A(u_h;v_h)=L(v_h) \quad \forall v_h\in V_h \label{eq:fine-scale-reference-problem}
\end{equation}
where \(A(\cdot;\cdot)\) is nonlinear in the first argument and linear in the second, while \(L\) is linear. We assume that \eqref{eq:fine-scale-reference-problem} is uniquely solvable and that the constrained fine-scale problems below are well posed. The structural conditions used for the model diffusion problem are stated in Section \ref{sec:heterogeneous-nonlinear-diffusion}.

\paragraph{Coarse--Fine Splitting.}
\label{par:coarse-fine-splitting}
Let \(\Pi_H:V_h\to V_H\) be a projection onto the coarse space \(V_H\), so that
\begin{equation}
\Pi_H v_H=v_H \quad \forall v_H\in V_H \label{eq:projection-identity-on-coarse-space}
\end{equation}
We assume that \(\Pi_H\) is stable in the energy norm,
\begin{equation}
\|\Pi_H v_h\|_V\le C_{\Pi}\|v_h\|_V \quad \forall v_h\in V_h \label{eq:projection-stability}
\end{equation}
with \(C_{\Pi}\) independent of the mesh sizes, and compatible with the local patch restrictions used in Section \ref{sec:localized-reconstruction}. Concrete admissible choices of \(\Pi_H\) are discussed in the remark below. The local stability constants needed by the perturbation argument are stated in Section \ref{sec:error-analysis}. We set
\begin{equation}
V_f=\operatorname{ker}\Pi_H \label{eq:fine-scale-space}
\end{equation}
as the fine-scale space. Then every \(v_h\in V_h\) has the decomposition
\begin{equation}
v_h=\Pi_H v_h +(v_h-\Pi_H v_h) \quad \Pi_H v_h\in V_H \quad v_h-\Pi_H v_h\in V_f \label{eq:coarse-fine-decomposition}
\end{equation}
where the first term is coarse and the second term belongs to \(V_f\). Thus \(V_f\) is the space of fine-scale functions that are invisible to the coarse projection.

\paragraph{Admissible Projections and Partitions of Unity.}
\label{par:admissible-choices}
The construction requires \(\Pi_H\) to be a projection satisfying the energy stability \eqref{eq:projection-stability}; the specific operator is otherwise unconstrained. Possible choices are, e.g., the \(L^2\)-orthogonal projection onto \(V_H\), whose \(H^1\)-stability on shape-regular, quasi-uniform meshes is classical \cite{BPS02}, or the Scott--Zhang quasi-interpolation operator \cite{ScottZhang1990}, which is a projection onto \(V_H\) that is both energy-stable and local. Similarly, the localized reconstruction of Section \ref{sec:localized-reconstruction} uses only a partition of unity subordinate to the coarse mesh; natural choices are the \(P_1\) coarse vertex hat functions \(\varphi_{H,z}=\Lambda_z\), whose active regions are the nodal stars, or piecewise-constant element indicators. The numerical experiments in Section \ref{sec:numerical-experiments} use the \(L^2\)-projection together with vertex-patch hats.

\paragraph{Global Fine-Scale Correction.}
\label{par:global-fine-scale-correction}
For each \(v_H\in V_H\), define \(v_f(v_H)\in V_f\) as the solution of
\begin{equation}
A(v_H+v_f(v_H);w_f)=L(w_f) \quad \forall w_f\in V_f \label{eq:global-fine-scale-correction}
\end{equation}
which enforces the residual equation on \(V_f\). The corresponding reconstruction map \(M_h:V_H\to V_h\) is
\begin{equation}
M_h(v_H)=v_H+v_f(v_H) \label{eq:global-reconstruction-map}
\end{equation}
and its image is
\begin{equation}
\mathcal{M}_h=M_h(V_H)\subset V_h \label{eq:global-nonlinear-manifold}
\end{equation}
which is the nonlinear fine-scale manifold. By construction, the residual of \(M_h(v_H)\) vanishes on \(V_f\) for every coarse state \(v_H\).

\begin{lem}[Coarse parametrization and tangent complement]
\label{lem:global-coarse-parametrization}
For every \(v_H\in V_H\), the global reconstruction satisfies
\begin{equation}
\Pi_HM_h(v_H)=v_H \label{eq:global-coarse-parametrization}
\end{equation}
If \(M_h\) is differentiable at \(v_H\), then
\begin{equation}
\Pi_HD M_h(v_H)[\delta v_H]=\delta v_H \quad \forall \delta v_H\in V_H \label{eq:global-tangent-coarse-parametrization}
\end{equation}
Consequently, the tangent space is a complement of \(V_f\),
\begin{equation}
V_h=V_f\oplus T_{M_h(v_H)}\mathcal M_h \label{eq:global-tangent-direct-sum}
\end{equation}
where \(T_{M_h(v_H)}\mathcal M_h\) is defined in \eqref{eq:global-tangent-space}.
\end{lem}
\begin{proof}
Since \(v_f(v_H)\in V_f=\operatorname{ker}\Pi_H\),
\begin{equation}
\Pi_HM_h(v_H)=\Pi_Hv_H+\Pi_Hv_f(v_H)=v_H \label{eq:global-coarse-parametrization-proof}
\end{equation}
Differentiating \eqref{eq:global-coarse-parametrization} gives \eqref{eq:global-tangent-coarse-parametrization}. If \(z_h\in V_f\cap T_{M_h(v_H)}\mathcal M_h\), then \(z_h=D M_h(v_H)[\delta v_H]\) for some \(\delta v_H\in V_H\), and \eqref{eq:global-tangent-coarse-parametrization} gives
\begin{equation}
\delta v_H=\Pi_H z_h=0 \label{eq:global-direct-sum-injectivity-proof}
\end{equation}
Thus the intersection is trivial. Since \(V_h=V_H\oplus V_f\) by \eqref{eq:coarse-fine-decomposition} and \(D M_h(v_H)\) is injective by \eqref{eq:global-tangent-coarse-parametrization}, the dimensions add up and \eqref{eq:global-tangent-direct-sum} follows.
\end{proof}

\paragraph{Exact Representation.}
\label{par:exact-representation}
Let \(u_h\) solve \eqref{eq:fine-scale-reference-problem} and set \(u_H=\Pi_H u_h\). Then \(u_h-u_H\in V_f\), and testing \eqref{eq:fine-scale-reference-problem} with \(w_f\in V_f\) gives
\begin{equation}
A(u_H+(u_h-u_H);w_f)=L(w_f) \quad \forall w_f\in V_f \label{eq:exact-representation-fine-scale-equation}
\end{equation}
which is the constrained correction problem with coarse state \(u_H\). By uniqueness of that problem,
\begin{equation}
u_h-u_H=v_f(u_H) \quad u_h=M_h(u_H) \label{eq:exact-global-manifold-representation}
\end{equation}
so the exact fine-scale solution belongs to \(\mathcal{M}_h\). The localized method in Section \ref{sec:localized-reconstruction} approximates this map rather than the solution itself.

\paragraph{Tangent Reconstruction.}
\label{par:tangent-reconstruction}
Assume that \(v_f\) is differentiable at \(v_H\). For a coarse direction \(\delta v_H\in V_H\), the tangent reconstruction is
\begin{equation}
D M_h(v_H)[\delta v_H]=\delta v_H+D v_f(v_H)[\delta v_H] \label{eq:global-tangent-reconstruction}
\end{equation}
where \(D v_f(v_H)[\delta v_H]\in V_f\) solves the linearized fine-scale problem
\begin{equation}
A'(v_H+v_f(v_H))[\delta v_H+D v_f(v_H)[\delta v_H],w_f]=0 \quad \forall w_f\in V_f \label{eq:global-tangent-correction}
\end{equation}
and \(A'(u)[z,w]\) denotes the derivative of \(A(\cdot;w)\) at \(u\) in the direction \(z\). The tangent space to \(\mathcal{M}_h\) at \(M_h(v_H)\) is
\begin{equation}
T_{M_h(v_H)}\mathcal{M}_h=\{D M_h(v_H)[\delta v_H]:\delta v_H\in V_H\} \label{eq:global-tangent-space}
\end{equation}
By Lemma \ref{lem:global-coarse-parametrization}, this tangent space contains exactly one lift of each coarse direction and is complementary to \(V_f\).

\paragraph{Graph Geometry.}
\label{par:global-graph-geometry}
The map \(M_h(v_H)=v_H+v_f(v_H)\) identifies the ideal manifold as a nonlinear graph over the coarse space. Equivalently, it is a section of the affine fibration induced by \(\Pi_H\), selecting one representative in each fine-scale fiber \(v_H+V_f\). Figure \ref{fig:graph-manifold-tangent} illustrates this geometry and the associated tangent lift.

\begin{figure}[htbp]
\centering
\includegraphics[width=0.90\textwidth]{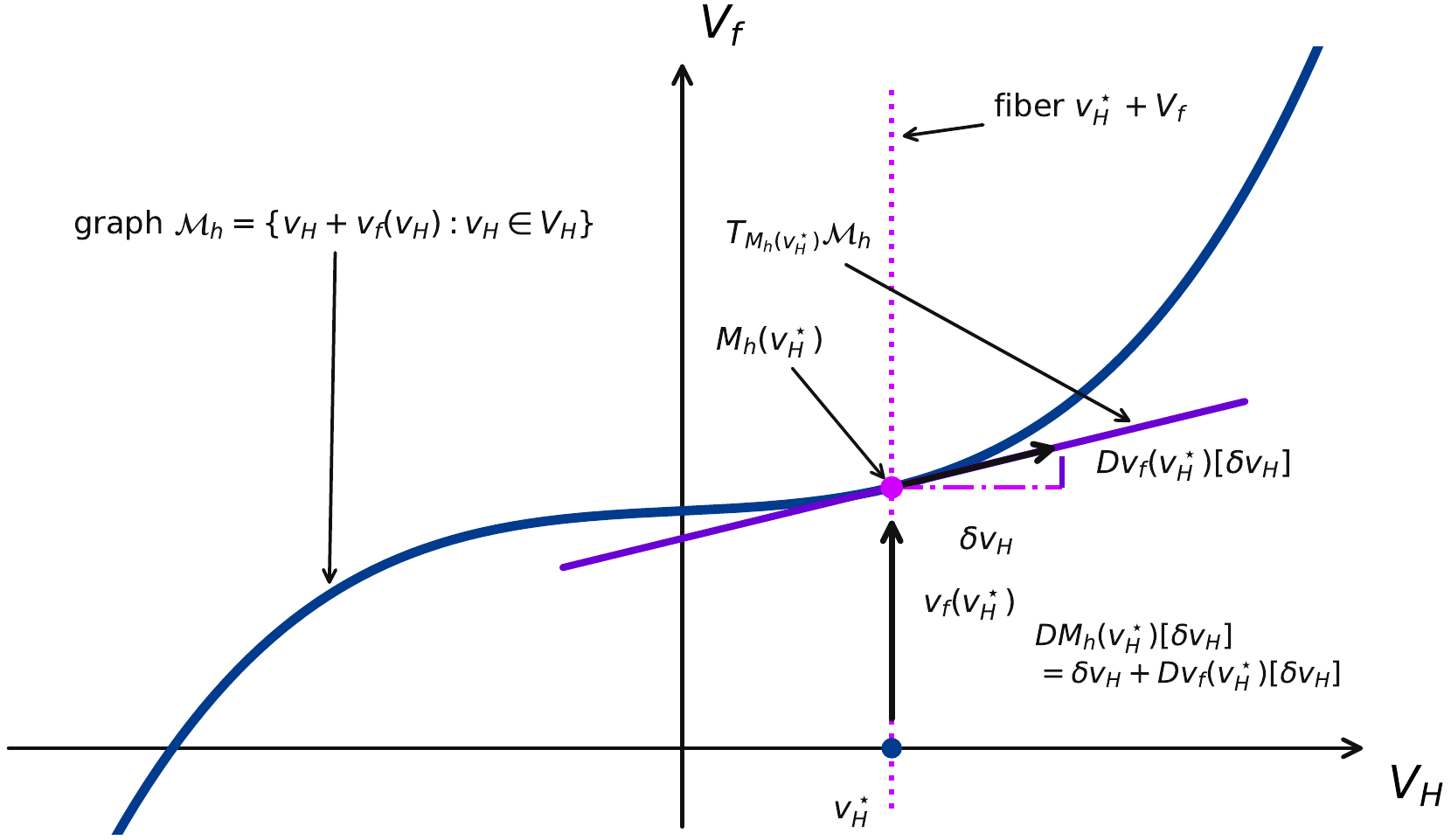}
\caption{Schematic illustration of the ideal multiscale manifold as a nonlinear graph over the coarse space \(V_H\). Each coarse state \(v_H\) determines a fine-scale fiber \(v_H+V_f\), and \(M_h(v_H)=v_H+v_f(v_H)\) selects one point on this fiber. The tangent space at \(M_h(v_H^\star)\) is generated by vectors \(\delta v_H+Dv_f(v_H^\star)[\delta v_H]\).}
\label{fig:graph-manifold-tangent}
\end{figure}

\paragraph{Ideal Manifold Equation.}
\label{par:ideal-manifold-equation}
The ideal coarse problem is to find \(u_H\in V_H\) such that
\begin{equation}
A(M_h(u_H);D M_h(u_H)[\delta v_H])=L(D M_h(u_H)[\delta v_H]) \quad \forall \delta v_H\in V_H \label{eq:global-manifold-problem}
\end{equation}
or equivalently, with \(u_h=M_h(u_H)\),
\begin{equation}
A(u_h;z_h)=L(z_h) \quad \forall z_h\in T_{u_h}\mathcal{M}_h \label{eq:global-tangent-residual-equation}
\end{equation}

\begin{prop}[Equivalence with the fine-scale problem]
\label{prop:ideal-manifold-equivalence}
Assume that the constrained correction problem \eqref{eq:global-fine-scale-correction} is uniquely solvable and that \(M_h\) is differentiable at the relevant coarse states. If \(\widehat u_H\in V_H\) solves \eqref{eq:global-manifold-problem}, then \(\widehat u_h=M_h(\widehat u_H)\) solves the fine-scale problem \eqref{eq:fine-scale-reference-problem}. Conversely, if \(u_h\) solves \eqref{eq:fine-scale-reference-problem} and \(u_H=\Pi_Hu_h\), then \(u_h=M_h(u_H)\) and \(u_H\) solves \eqref{eq:global-manifold-problem}.
\end{prop}
\begin{proof}
Let \(\widehat u_H\) solve \eqref{eq:global-manifold-problem} and set \(\widehat u_h=M_h(\widehat u_H)\). The definition of \(M_h\) gives
\begin{equation}
A(\widehat u_h;w_f)=L(w_f) \quad \forall w_f\in V_f \label{eq:ideal-equivalence-fine-residual}
\end{equation}
and \eqref{eq:global-manifold-problem} gives
\begin{equation}
A(\widehat u_h;z_h)=L(z_h) \quad \forall z_h\in T_{\widehat u_h}\mathcal M_h \label{eq:ideal-equivalence-tangent-residual}
\end{equation}
By \eqref{eq:global-tangent-direct-sum}, every \(v_h\in V_h\) can be written uniquely as \(v_h=w_f+z_h\) with \(w_f\in V_f\) and \(z_h\in T_{\widehat u_h}\mathcal M_h\). Since the residual is linear in the test function,
\begin{equation}
A(\widehat u_h;v_h)=L(v_h) \quad \forall v_h\in V_h \label{eq:ideal-equivalence-full-residual}
\end{equation}
Thus \(\widehat u_h\) solves \eqref{eq:fine-scale-reference-problem}. The converse follows from \eqref{eq:exact-global-manifold-representation} and by testing the fine-scale equation with the tangent functions \(D M_h(u_H)[\delta v_H]\).
\end{proof}

\paragraph{Mean-Value Identity.}
\label{par:mean-value-identity}
The tangent formulation is needed because \(\mathcal{M}_h\) is not a linear space. If \(M_h\) is differentiable along the segment from \(w_H\) to \(v_H\), then
\begin{equation}
M_h(v_H)-M_h(w_H)=\int_0^1 D M_h(w_H+s(v_H-w_H))[v_H-w_H] \,ds \label{eq:mean-value-identity}
\end{equation}
which replaces the linear-space difference argument used in standard Galerkin orthogonality.

%% file: sections/localized_reconstruction_v4.tex

\section{Localized Reconstruction}
\label{sec:localized-reconstruction}

This section replaces the global correction map \(M_h\) from Section \ref{sec:global-fine-scale-manifold} by a computable localized map. The construction solves nonlinear residual equations on oversampled patches, blends only the active parts of the local corrections by a coarse partition of unity, and projects the blended correction back to \(V_f\). At this stage the only required assumptions are well-posedness and differentiability of the local patch maps at the states where they are evaluated.

\paragraph{Patch Spaces and Local Coarse Data.}
\label{par:patch-spaces}
Let \(\{\varphi_{H,i}\}_{i\in\mathcal I}\) be a partition of unity subordinate to the coarse mesh \(\mathcal T_H\), with each support \(\operatorname{supp}\varphi_{H,i}\) a union of coarse elements (see the remark in Section \ref{sec:global-fine-scale-manifold} for admissible choices). Patches are built from coarse-element neighborhoods. For a subdomain \(\omega\subset\Omega\) that is a union of coarse elements, one coarse-element layer is added by
\begin{subequations}\label{eq:element-neighborhood}
\begin{align}
\mathsf{N}(\omega)&=\operatorname{int}\Bigl(\bigcup\{\overline{T}: T\in\mathcal T_H,\ \overline{T}\cap\overline{\omega}\ne\emptyset\}\Bigr),\\ \mathsf{N}^{0}(\omega)&=\omega,\\ \mathsf{N}^{\ell}(\omega)&=\mathsf{N}\bigl(\mathsf{N}^{\ell-1}(\omega)\bigr)\ \ (\ell\ge1) 
\end{align}
\end{subequations}
The active region of patch \(i\) is
\begin{equation}
\omega_i^{0}=\operatorname{int}\bigl(\operatorname{supp}\varphi_{H,i}\bigr) \label{eq:partition-support-region}
\end{equation}
and the oversampled patch of order \(\ell\ge0\) is its \(\ell\)-layer coarse-element neighborhood,
\begin{equation}
\omega_i^{\ell}=\mathsf{N}^{\ell}(\omega_i^{0}) \label{eq:oversampled-patch}
\end{equation}
Each \(\omega_i^{\ell}\) is a union of coarse elements with \(\omega_i^{0}\subset\omega_i^{\ell}\subset\omega_i^{\ell+1}\), and \(\ell\) counts the coarse-element layers added beyond the active region. We fix a basis \(\{\Phi_j\}_{j=1}^{N_H}\) of \(V_H\) and define the local coarse index set by
\begin{equation}
\mathcal I_i^{\ell}=\{j\in\{1,\ldots,N_H\}:\operatorname{supp}\Phi_j\cap\omega_i^{\ell}\ne\emptyset\} \label{eq:local-coarse-index-set}
\end{equation}
If \(v_H=\sum_{j=1}^{N_H}c_j\Phi_j\), then the local restriction operator is
\begin{equation}
P_i v_H=(c_j)_{j\in\mathcal I_i^{\ell}}\in\mathbb R^{n_{H,i}} \quad n_{H,i}=|\mathcal I_i^{\ell}| \label{eq:local-coarse-restriction-operator}
\end{equation}
Thus \(P_i v_H\) determines \(v_H\) on \(\omega_i^{\ell}\). The local fine-scale space is
\begin{equation}
V_{f,i}^{\ell}=\bigl\{v_f\in V_f: \operatorname{supp} v_f\subset \overline{\omega_i^{\ell}},\ v_f=0\ \text{on}\ \partial\omega_i^{\ell}\bigr\} \label{eq:local-fine-scale-space}
\end{equation}
so that the local corrections satisfy homogeneous Dirichlet conditions on the artificial patch boundary. We write \(E_i^{\operatorname{ext}}\) for extension by zero from \(\omega_i^{\ell}\) to the global fine grid. The partition of unity is active only on the star \(\omega_i^{0}\subset\omega_i^{\ell}\) defined in \eqref{eq:partition-support-region}. Let \(R_i\) denote restriction to \(\omega_i^{0}\), and let \(W_{f,i}^{0}\) be the corresponding restricted output space,
\begin{equation}
R_i\colon V_{f,i}^{\ell}\to W_{f,i}^{0} \quad W_{f,i}^{0}=R_i V_{f,i}^{\ell} \label{eq:restricted-output-space}
\end{equation}
We write \(E_{i,0}^{\operatorname{ext}}\) for extension by zero from \(\omega_i^{0}\) to the global fine grid.

\paragraph{Local Patch Problems.}
\label{par:local-patch-problems}
For \(v_H\in V_H\), the local correction \(v_{f,i}^{\ell}(P_i v_H)\in V_{f,i}^{\ell}\) is defined by
\begin{align}
A_{\omega_i^{\ell}}\bigl(v_H+v_{f,i}^{\ell}(P_i v_H);w_{f,i}\bigr)
&=L_{\omega_i^{\ell}}(w_{f,i}) \quad \forall w_{f,i}\in V_{f,i}^{\ell} \label{eq:patch-correction-problem}
\end{align}
where \(A_{\omega_i^{\ell}}\) and \(L_{\omega_i^{\ell}}\) are the restrictions of the global forms to the patch. Although \(v_H\) appears in the residual, only its patch restriction is used, and this restriction is determined by \(P_i v_H\) through \eqref{eq:local-coarse-restriction-operator}. This is why the local map can be viewed as a finite-dimensional map of \(n_{H,i}\) coarse variables.

\paragraph{Localized Manifold Map.}
\label{par:localized-manifold-map}
Only the restriction of the patch correction to \(\omega_i^{0}\) enters the global blend. Indeed,
\begin{align}
\varphi_{H,i}E_i^{\operatorname{ext}}v_{f,i}^{\ell}(P_i v_H)
&=\varphi_{H,i}E_{i,0}^{\operatorname{ext}}R_i v_{f,i}^{\ell}(P_i v_H) \quad \forall v_H\in V_H \label{eq:partition-support-restriction}
\end{align}
The localized reconstruction map is therefore
\begin{align}
M_h^{\ell}(v_H)
&=v_H+(I-\Pi_H)I_h\Bigl(\sum_{i\in\mathcal I}\varphi_{H,i}E_{i,0}^{\operatorname{ext}}R_i v_{f,i}^{\ell}(P_i v_H)\Bigr) \label{eq:localized-reconstruction-map}
\end{align}
where \(I_h\) denotes a stable fine-grid interpolation or projection. The product \(\varphi_{H,i}E_{i,0}^{\operatorname{ext}}R_i v_{f,i}^{\ell}\) is generally not an element of \(V_h\), and \(I_h\) places the partition-of-unity blend in the fine space. The projection \(I-\Pi_H\) is essential because the blended correction need not belong to \(V_f\). The localized nonlinear trial manifold is
\begin{equation}
\mathcal{M}_h^{\ell}=M_h^{\ell}(V_H)\subset V_h \label{eq:localized-trial-manifold}
\end{equation}
which is the computable counterpart of the ideal manifold \(\mathcal{M}_h\).

\paragraph{Localized Tangent Map.}
\label{par:localized-tangent-map}
For a coarse direction \(\delta v_H\in V_H\), differentiation of \eqref{eq:localized-reconstruction-map} gives
\begin{align}
D M_h^{\ell}(v_H)[\delta v_H]
&=\delta v_H+(I-\Pi_H)I_h\Bigl(\sum_{i\in\mathcal I}\varphi_{H,i}E_{i,0}^{\operatorname{ext}}R_i D v_{f,i}^{\ell}(P_i v_H)[P_i\delta v_H]\Bigr) \label{eq:localized-tangent-reconstruction}
\end{align}
The local tangent correction \(D v_{f,i}^{\ell}(P_i v_H)[P_i\delta v_H]\in V_{f,i}^{\ell}\) is obtained by differentiating \eqref{eq:patch-correction-problem}. It solves
\begin{align}
A_{\omega_i^{\ell}}'\bigl(v_H+v_{f,i}^{\ell}(P_i v_H)\bigr)\bigl[\delta v_H+D v_{f,i}^{\ell}(P_i v_H)[P_i\delta v_H],w_{f,i}\bigr]
&=0 \quad \forall w_{f,i}\in V_{f,i}^{\ell} \label{eq:localized-tangent-patch-problem}
\end{align}
where the derivative is taken with respect to the first argument of \(A_{\omega_i^{\ell}}\). As for the state correction, only the restriction of this tangent correction to \(\omega_i^{0}\) is used in the global tangent reconstruction.

\begin{lem}[Localized coarse parametrization]
\label{lem:localized-coarse-parametrization}
For every \(v_H\in V_H\), the localized reconstruction satisfies
\begin{equation}
\Pi_HM_h^{\ell}(v_H)=v_H \label{eq:localized-coarse-parametrization}
\end{equation}
If \(M_h^{\ell}\) is differentiable at \(v_H\), then
\begin{equation}
\Pi_HD M_h^{\ell}(v_H)[\delta v_H]=\delta v_H \quad \forall \delta v_H\in V_H \label{eq:localized-tangent-coarse-parametrization}
\end{equation}
Consequently,
\begin{equation}
V_h=V_f\oplus T_{M_h^{\ell}(v_H)}\mathcal M_h^{\ell} \label{eq:localized-tangent-direct-sum}
\end{equation}
whenever the tangent space is defined.
\end{lem}
\begin{proof}
The added correction in \eqref{eq:localized-reconstruction-map} lies in \(V_f\) because of the factor \(I-\Pi_H\). Therefore,
\begin{equation}
\Pi_HM_h^{\ell}(v_H)=\Pi_Hv_H=v_H \label{eq:localized-coarse-parametrization-proof}
\end{equation}
Differentiating \eqref{eq:localized-coarse-parametrization} gives \eqref{eq:localized-tangent-coarse-parametrization}. The direct-sum statement follows as in Lemma \ref{lem:global-coarse-parametrization}.
\end{proof}

\paragraph{Evaluation Procedure.}
\label{par:localized-evaluation-procedure}
The deterministic localized reconstruction can be evaluated patchwise. For a given coarse state \(v_H\), one solves \eqref{eq:patch-correction-problem} independently on all patches, restricts the local outputs to \(\omega_i^0\), blends the restricted outputs with \(\varphi_{H,i}\), applies \(I_h\), and finally projects with \(I-\Pi_H\). For tangent actions, one reuses the local state solutions, solves the linearized patch problems \eqref{eq:localized-tangent-patch-problem} for the required local coarse directions, and assembles the result by \eqref{eq:localized-tangent-reconstruction}.

\paragraph{Use in the Coarse Problem.}
\label{par:localized-use-in-coarse-problem}
The localized map \(M_h^{\ell}\) and its tangent action \(D M_h^{\ell}\) define the deterministic computable method. The corresponding tangent-space coarse equation is stated in Section \ref{sec:computable-coarse-problems}; the only change relative to the ideal equation in Section \ref{sec:global-fine-scale-manifold} is the replacement of \(M_h\) by \(M_h^{\ell}\).

%% file: sections/network_interpolation_v4.tex

\section{Network Interpolation of Local Patch Maps}
\label{sec:network-interpolation}

The localized patch-solved method from Section \ref{sec:localized-reconstruction} is already computable. This section describes an optional acceleration in which the restricted local patch-output operators are replaced by learned interpolants. The restriction is important: since the global reconstruction only uses the correction on the active support \(\omega_i^0\), no network output is required on the oversampling layer \(\omega_i^{\ell}\setminus\omega_i^0\). Because the coarse equation uses tangent test functions, the relevant approximation target is the patch map together with its derivative.

\paragraph{Restricted Patch Operators.}
\label{par:restricted-patch-operators}
For patch \(i\), let \(R_i\) and \(W_{f,i}^{0}\) be defined as in Section \ref{sec:localized-reconstruction}. After choosing a local basis \(\{\psi_{i,r}^{0}\}_{r=1}^{n_{0,i}}\) of \(W_{f,i}^{0}\), we identify the restricted output space with coefficient vectors by
\begin{equation}
w_i^0=\sum_{r=1}^{n_{0,i}}\gamma_{i,r}\psi_{i,r}^{0} \quad \longleftrightarrow \quad \gamma_i=(\gamma_{i,1},\ldots,\gamma_{i,n_{0,i}})\in\mathbb R^{n_{0,i}} \label{eq:restricted-output-coordinates}
\end{equation}
The operator needed by the global reconstruction is
\begin{align}
\mathcal G_{f,i}^{\ell}&\colon \mathbb R^{n_{H,i}}\to W_{f,i}^{0} \label{eq:local-patch-map-domain}\\
\mathcal G_{f,i}^{\ell}(P_i v_H)&=R_i v_{f,i}^{\ell}(P_i v_H) \quad \forall v_H\in V_H \label{eq:local-patch-map}
\end{align}
Thus the full patch problem on \(\omega_i^{\ell}\) may be used to generate the local output, but values outside \(\omega_i^{0}\) are not part of the interpolated data.

\paragraph{Network Interpolants.}
\label{par:network-interpolants}
For each patch, we replace \(\mathcal G_{f,i}^{\ell}\) by a learned interpolant with parameter vector \(\theta_i\),
\begin{equation}
\mathcal N_{f,i}^{\ell,\theta_i}\colon \mathbb R^{n_{H,i}}\to \mathbb R^{n_{0,i}} \label{eq:network-patch-map}
\end{equation}
where the output is interpreted as nodal values of a fine finite element function on \(\omega_i^{0}\). We define the learned restricted correction by
\begin{equation}
g_{f,i}^{\ell,\theta_i}(P_i v_H)=\mathcal N_{f,i}^{\ell,\theta_i}(P_i v_H) \label{eq:learned-patch-correction}
\end{equation}
With \(\theta=(\theta_i)_{i\in\mathcal I}\), the learned global reconstruction is
\begin{align}
M_{h,\theta}^{\ell}(v_H)
&=v_H+(I-\Pi_H)I_h
\Bigl(\sum_{i\in\mathcal I}\varphi_{H,i}E_{i,0}^{\operatorname{ext}}
       g_{f,i}^{\ell,\theta_i}(P_i v_H)\Bigr) \label{eq:learned-reconstruction-map}
\end{align}
It is a network-interpolated approximation of \(M_h^{\ell}\), while \(M_h^{\ell}\) is the localized approximation of the ideal map \(M_h\). As in Lemma \ref{lem:localized-coarse-parametrization}, the projection \(I-\Pi_H\) gives
\begin{equation}
\Pi_HM_{h,\theta}^{\ell}(v_H)=v_H \quad \forall v_H\in V_H \label{eq:learned-coarse-parametrization}
\end{equation}
In particular, the network output is not required to lie in \(W_{f,i}^{0}=R_iV_{f,i}^{\ell}\): arbitrary nodal values on \(\omega_i^{0}\) are admissible, since the projection \(I-\Pi_H\) restores membership in \(V_f\) after the partition-of-unity blend.

\paragraph{Learned Tangent Reconstruction.}
\label{par:learned-tangent-reconstruction}
The tangent map is obtained by differentiating the learned restricted patch maps. For \(\delta v_H\in V_H\),
\begin{align}
D M_{h,\theta}^{\ell}(v_H)[\delta v_H]
&=\delta v_H+(I-\Pi_H)I_h
\Bigl(\sum_{i\in\mathcal I}\varphi_{H,i}E_{i,0}^{\operatorname{ext}}
       D\mathcal N_{f,i}^{\ell,\theta_i}(P_i v_H)[P_i\delta v_H]\Bigr) \label{eq:learned-tangent-reconstruction}
\end{align}
This gives tangent test functions without solving linearized patch problems online. The derivatives of the network maps can be evaluated by automatic differentiation. The derivative is not merely an implementation detail: tangent-space consistency requires accuracy of both \(\mathcal N_{f,i}^{\ell,\theta_i}\) and \(D\mathcal N_{f,i}^{\ell,\theta_i}\) on the admissible local parameter set.

\paragraph{Training Data and Loss.}
\label{par:training-data-and-loss}
A minimal local data set consists of local coarse states and the corresponding restricted patch corrections,
\begin{equation}
\mathcal D_i^{\ell}=\bigl\{(P_i v_H^m,R_i v_{f,i}^{\ell}(P_i v_H^m))\bigr\}_{m=1}^{N_{\operatorname{train}}} \label{eq:patch-training-data}
\end{equation}
When tangent data are available, let \(\rho_{i,q}\in\mathbb R^{n_{H,i}}\), \(q=1,\ldots,Q_i\), denote local coarse directions used to sample derivative information. A derivative-enhanced local loss is
\begin{align}
\mathcal L_i(\theta_i)
&=\sum_{m=1}^{N_{\operatorname{train}}}\bigl\|\mathcal N_{f,i}^{\ell,\theta_i}(P_i v_H^m)-\mathcal G_{f,i}^{\ell}(P_i v_H^m)\bigr\|_{H^1(\omega_i^{0})}^2 \label{eq:network-training-loss-state}\\
&\quad+\lambda\sum_{m=1}^{N_{\operatorname{train}}}\sum_{q=1}^{Q_i}\bigl\|D\mathcal N_{f,i}^{\ell,\theta_i}(P_i v_H^m)[\rho_{i,q}]-D\mathcal G_{f,i}^{\ell}(P_i v_H^m)[\rho_{i,q}]\bigr\|_{H^1(\omega_i^{0})}^2 \label{eq:network-training-loss-tangent}
\end{align}
with \(\lambda\ge0\). Choosing \(\lambda=0\) gives the state-only loss. Section \ref{sec:error-analysis} does not depend on a particular training algorithm. It uses either the realized reconstruction error of the learned map or the reduced residual obtained after solving the learned coarse problem; tangent approximation errors enter when one compares learned and deterministic residuals or uses a tangent surrogate that is not the derivative of the same learned reconstruction.

%% file: sections/computable_coarse_problems_v4.tex

\section{Computable Coarse Problems}
\label{sec:computable-coarse-problems}

This section gives the finite-dimensional problems solved online. Sections \ref{sec:localized-reconstruction} and \ref{sec:network-interpolation} define the reconstruction maps; here we state the tangent-space coarse equation, the residual used in computation, and the work required for one residual evaluation.

\paragraph{Generic Reconstruction.}
\label{par:computable-generic-reconstruction}
We use \(\widetilde M_h^{\ell}\) for either the patch-solved or the network-interpolated reconstruction,
\begin{equation}
\widetilde M_h^{\ell}=M_h^{\ell} \quad \text{or} \quad \widetilde M_h^{\ell}=M_{h,\theta}^{\ell} \label{eq:generic-reconstruction-choices}
\end{equation}
The corresponding restricted local output is denoted by \(\widetilde g_i^{\ell}\). For \(\mu\in\mathbb R^{n_{H,i}}\),
\begin{align}
\widetilde g_i^{\ell}(\mu)
&=\begin{cases}
R_i v_{f,i}^{\ell}(\mu), & \text{for the patch-solved method}\\
\mathcal N_{f,i}^{\ell,\theta_i}(\mu), & \text{for the network-interpolated method}
\end{cases} \label{eq:generic-restricted-output}
\end{align}
With this notation both reconstructions have the common form
\begin{align}
\widetilde M_h^{\ell}(v_H)
&=v_H+(I-\Pi_H)I_h\Bigl(\sum_{i\in\mathcal I}\varphi_{H,i}E_{i,0}^{\operatorname{ext}}\widetilde g_i^{\ell}(P_i v_H)\Bigr) \label{eq:generic-reconstruction-map}
\end{align}
and the tangent action is
\begin{align}
D\widetilde M_h^{\ell}(v_H)[\delta v_H]
&=\delta v_H+(I-\Pi_H)I_h\Bigl(\sum_{i\in\mathcal I}\varphi_{H,i}E_{i,0}^{\operatorname{ext}}D\widetilde g_i^{\ell}(P_i v_H)[P_i\delta v_H]\Bigr) \label{eq:generic-tangent-map}
\end{align}
For the patch-solved method, \(D\widetilde g_i^{\ell}\) is obtained from the linearized patch problem \eqref{eq:localized-tangent-patch-problem}. For the network-interpolated method, \(D\widetilde g_i^{\ell}\) is the derivative of the learned local map. In both cases,
\begin{equation}
\Pi_H\widetilde M_h^{\ell}(v_H)=v_H \quad \Pi_HD\widetilde M_h^{\ell}(v_H)[\delta v_H]=\delta v_H \label{eq:generic-coarse-parametrization}
\end{equation}
whenever the derivative is defined.

\paragraph{Tangent-Space Coarse Equation.}
\label{par:tangent-space-coarse-equation}
The computable coarse problem is the tangent-space Galerkin equation on the chosen localized manifold. It reads: find \(\widetilde u_H^{\ell}\in V_H\) such that
\begin{align}
A\bigl(\widetilde M_h^{\ell}(\widetilde u_H^{\ell});D\widetilde M_h^{\ell}(\widetilde u_H^{\ell})[\delta v_H]\bigr)
&=L\bigl(D\widetilde M_h^{\ell}(\widetilde u_H^{\ell})[\delta v_H]\bigr) \quad \forall \delta v_H\in V_H \label{eq:computable-coarse-problem}
\end{align}
The final fine-scale approximation is
\begin{equation}
\widetilde u_{\mathrm{ms}}^{\ell}=\widetilde M_h^{\ell}(\widetilde u_H^{\ell}) \label{eq:generic-fine-scale-approximation}
\end{equation}
Thus \(\widetilde M_h^{\ell}=M_h^{\ell}\) gives the deterministic patch-solved method, while \(\widetilde M_h^{\ell}=M_{h,\theta}^{\ell}\) gives the optional network-interpolated method. In the analysis and numerical experiments the nonlinear algebraic problem is considered on an admissible set where the reconstruction and tangent maps are defined.

\paragraph{Reduced Energy Interpretation.}
\label{par:reduced-energy-interpretation}
If the fine-scale problem is the Euler--Lagrange equation of an energy \(\mathcal E_h:V_h\to\mathbb R\), with
\begin{equation}
D\mathcal E_h(u_h)[v_h]=A(u_h;v_h)-L(v_h) \quad \forall u_h,v_h\in V_h \label{eq:energy-derivative-assumption}
\end{equation}
then \eqref{eq:computable-coarse-problem} is the stationarity condition for the reduced energy
\begin{equation}
\mathcal J_H^{\ell}(v_H)=\mathcal E_h(\widetilde M_h^{\ell}(v_H)) \quad v_H\in V_H \label{eq:reduced-energy-functional}
\end{equation}
Indeed, the chain rule gives
\begin{align}
D\mathcal J_H^{\ell}(v_H)[\delta v_H]
&=D\mathcal E_h(\widetilde M_h^{\ell}(v_H))[D\widetilde M_h^{\ell}(v_H)[\delta v_H]] \label{eq:reduced-energy-chain-rule}\\
&=A\bigl(\widetilde M_h^{\ell}(v_H);D\widetilde M_h^{\ell}(v_H)[\delta v_H]\bigr)-L\bigl(D\widetilde M_h^{\ell}(v_H)[\delta v_H]\bigr) \label{eq:reduced-energy-derivative}
\end{align}
This interpretation is used for the model diffusion problem in Section \ref{sec:heterogeneous-nonlinear-diffusion} and is useful for designing nonlinear solvers.

\paragraph{Online Residual Evaluation.}
\label{par:online-residual-evaluation}
Let \(\{\Phi_j\}_{j=1}^{N_H}\) be the basis of \(V_H\) used in \eqref{eq:local-coarse-restriction-operator}. For a coefficient vector \(c\in\mathbb R^{N_H}\), define
\begin{equation}
v_H(c)=\sum_{k=1}^{N_H}c_k\Phi_k \label{eq:coarse-vector-to-function}
\end{equation}
and the local parameter vector
\begin{equation}
\mu_i(c)=P_i v_H(c) \quad \forall i\in\mathcal I \label{eq:local-parameter-vector}
\end{equation}
A residual evaluation consists of four steps. First, compute the restricted local outputs \(\widetilde g_i^{\ell}(\mu_i(c))\) independently for all patches. Second, assemble the global state \(\widetilde M_h^{\ell}(v_H(c))\) by \eqref{eq:generic-reconstruction-map}. Third, for each coarse basis direction \(\Phi_j\), evaluate the local tangent outputs \(D\widetilde g_i^{\ell}(\mu_i(c))[P_i\Phi_j]\). The block locality is
\begin{equation}
P_i\Phi_j=0 \quad \text{and hence} \quad D\widetilde g_i^{\ell}(\mu_i(c))[P_i\Phi_j]=0 \quad \text{if } j\notin\mathcal I_i^{\ell} \label{eq:local-tangent-block-sparsity}
\end{equation}
Fourth, assemble the tangent functions \(D\widetilde M_h^{\ell}(v_H(c))[\Phi_j]\) by \eqref{eq:generic-tangent-map} and test the residual. For the patch-solved method, the linearized patch operator can be assembled once per patch and reused for all local tangent right-hand sides associated with \(j\in\mathcal I_i^{\ell}\).

\paragraph{Algebraic Residual.}
\label{par:algebraic-residual}
The resulting coarse residual has components
\begin{align}
\mathcal R_j(c)
&=A\bigl(\widetilde M_h^{\ell}(v_H(c));D\widetilde M_h^{\ell}(v_H(c))[\Phi_j]\bigr)
  -L\bigl(D\widetilde M_h^{\ell}(v_H(c))[\Phi_j]\bigr),
  \quad j=1,\ldots,N_H \label{eq:coarse-residual-vector}
\end{align}
and the online solve is the nonlinear algebraic system
\begin{equation}
\mathcal R(c)=0 \label{eq:coarse-residual-system}
\end{equation}
This is the form used in implementation and in the error discussion.

Algorithm~\ref{alg:patch-solved-online-solve} summarizes the corresponding patch-solved online procedure. The matrix \(J_H\) denotes the Jacobian, a Jacobian approximation, or a matrix-free linearization used by the chosen nonlinear coarse solver. In the experiments of Section~\ref{sec:numerical-experiments} and in the extended algorithm of Appendix~\ref{sec:extended-patch-solved-algorithm} it is the \emph{Galerkin Jacobian} \(J_H=(D\widetilde M_h^\ell)^{\top}A'(\widetilde u_{\mathrm{ms}}^\ell)\,D\widetilde M_h^\ell\), obtained by dropping the curvature term \(\langle\mathcal R,D^2\widetilde M_h^\ell\rangle\) from the exact outer derivative; this inexact choice is derived and motivated in Section~\ref{sec:numerical-experiments}.

\begin{algorithm}[H]
\caption{Patch-solved online solve for the localized manifold method}
\label{alg:patch-solved-online-solve}
\begin{algorithmic}[1]
\State Precompute the patches \(\omega_i^\ell\), the active regions \(\omega_i^0\), the partition of unity \(\varphi_{H,i}\), the restriction maps \(P_i\) and \(R_i\), and the load contributions
\State Choose an initial coarse state \(u_H^{(0)}\in V_H\) and set \(n\gets0\)
\Repeat \Comment{coarse nonlinear iteration}
\ForAll{patches \(i\in\mathcal I\)} \Comment{independent nonlinear patch solves}
\State Initialize the local correction \(v_{f,i}\gets0\)
\Repeat \Comment{local Newton iteration}
\State Solve the constrained linearized patch system for \(\delta v_{f,i}\)
\State Update \(v_{f,i}\gets v_{f,i}+\delta v_{f,i}\)
\Until{the local patch residual is below tolerance}
\State Store the active output \(R_i v_{f,i}\)
\State Assemble the local tangent solves associated with \eqref{eq:localized-tangent-patch-problem}
\EndFor
\State Assemble \(u_{\operatorname{rec}}=M_h^\ell(u_H^{(n)})\) from \eqref{eq:localized-reconstruction-map}
\State Assemble the tangent actions \(D M_h^\ell(u_H^{(n)})[\Phi_j]\) for the active coarse directions
\State Compute the reduced residual vector \(r_H\) using \eqref{eq:coarse-residual-vector}
\If{\(\|r_H\|\) is below tolerance}
\State \Return \(u_{\mathrm{ms}}^\ell=u_{\operatorname{rec}}\)
\EndIf
\State Solve \(J_H\delta u_H=-r_H\) on the interior coarse degrees of freedom \Comment{Galerkin Jacobian; the curvature term \(\langle\mathcal R,D^2 M_h^\ell\rangle\) is dropped}
\State Update \(u_H^{(n+1)}\gets u_H^{(n)}+\delta u_H\) and \(n\gets n+1\)
\Until{the coarse iteration has converged}
\State \Return \(u_{\mathrm{ms}}^\ell=u_{\operatorname{rec}}\)
\end{algorithmic}
\end{algorithm}

\paragraph{Nonlinear Solver and Work Split.}
\label{par:nonlinear-solver-and-work-split}
The formulation of \(\mathcal R\) requires only first derivatives of the local patch maps. An exact Jacobian for Newton's method would differentiate \eqref{eq:coarse-residual-vector} once more and therefore involves second derivatives of the reconstruction. In practice one may instead use damped Newton with finite-difference Jacobian actions, quasi-Newton updates, or matrix-free Newton--Krylov iterations based on residual evaluations. In the patch-solved method, the online cost is dominated by independent nonlinear and linearized patch solves. In the network-interpolated method, those online solves are replaced by evaluations of \(\mathcal N_{f,i}^{\ell,\theta_i}\) and its derivative; the patch solves are moved to the offline data-generation stage.

%% file: sections/heterogeneous_nonlinear_diffusion_v3.tex

\section{Heterogeneous Nonlinear Diffusion}
\label{sec:heterogeneous-nonlinear-diffusion}

This section fixes the model problem used to verify the structural assumptions and to guide the numerical experiments. The construction in Sections \ref{sec:global-fine-scale-manifold}--\ref{sec:computable-coarse-problems} applies to more general nonlinear elliptic problems; here we specialize to a heterogeneous monotone diffusion law.

\paragraph{Model Setting.}
\label{par:model-setting}
We use the domain \(\Omega\) and the nested \(P_1\) Lagrange spaces \(V_H\subset V_h\) introduced in Section \ref{sec:global-fine-scale-manifold}; nonhomogeneous Dirichlet data can be treated by a standard lifting. These spaces are \(H^1\)-conforming, with \(V_h\subset H_0^1(\Omega)\) and energy norm
\begin{equation}
\|v_h\|_V=\|\nabla v_h\|_{\Omega} \label{eq:model-discrete-spaces}
\end{equation}
All structural statements below are used at the discrete level. Since \(V_h\) and the patch spaces are finite dimensional, differentiability and local Lipschitz continuity are understood on admissible subsets of these discrete spaces. Let \(f\in L^2(\Omega)\), let \(\alpha\ge 0\), and assume that \(a_{\varepsilon}\in L^{\infty}(\Omega)\) satisfies
\begin{equation}
0<a_0\le a_{\varepsilon}(x)\le a_1<\infty \quad \operatorname{a.e.}\ x\in\Omega \label{eq:model-coefficient-bounds}
\end{equation}

\paragraph{Energy and Weak Form.}
\label{par:model-energy}
The discrete energy is
\begin{align}
\mathcal E_h(v_h)
&=\int_{\Omega}
a_{\varepsilon}(x)
\biggl(
\frac{1}{2}|\nabla v_h|^2+
\frac{\alpha}{4}|\nabla v_h|^4
\biggr)\,dx
-\int_{\Omega}fv_h\,dx \label{eq:model-energy}
\end{align}
The nonlinear form and the load functional are
\begin{align}
A_h(u_h;v_h)
&=\int_{\Omega}
a_{\varepsilon}(x)
\bigl(1+\alpha|\nabla u_h|^2\bigr)
\nabla u_h\cdot\nabla v_h\,dx \label{eq:model-nonlinear-form}\\
F_h(v_h)
&=\int_{\Omega}fv_h\,dx \label{eq:model-load-functional}
\end{align}
The fine-scale reference problem is to find \(u_h\in V_h\) such that
\begin{equation}
A_h(u_h;v_h)=F_h(v_h) \quad \forall v_h\in V_h \label{eq:model-fine-scale-problem}
\end{equation}
This problem is the Euler--Lagrange equation for \(\mathcal E_h\), that is,
\begin{equation}
D\mathcal E_h(u_h)[v_h]=A_h(u_h;v_h)-F_h(v_h) \quad \forall u_h,v_h\in V_h \label{eq:model-energy-derivative}
\end{equation}

\paragraph{Structural Properties.}
\label{par:model-structural-properties}
The following proposition records the model-specific properties needed by the formulation and the abstract error estimate.
\begin{prop}[Structural properties]
\label{prop:model-structural-properties}
Assume \eqref{eq:model-coefficient-bounds}. Then \(A_h\) is strongly monotone, locally Lipschitz continuous, and differentiable in its first argument. More precisely, for all \(v_h,w_h\in V_h\),
\begin{equation}
A_h(v_h;v_h-w_h)-A_h(w_h;v_h-w_h)
\ge a_0\|v_h-w_h\|_V^2 \label{eq:model-strong-monotonicity}
\end{equation}
For every \(K>0\), define the admissible set
\begin{equation}
\mathcal B_K=\bigl\{v_h\in V_h:\|\nabla v_h\|_{L^{\infty}(\Omega)}\le K\bigr\} \label{eq:model-admissible-set}
\end{equation}
For all \(v_h,w_h\in\mathcal B_K\),
\begin{equation}
\|A_h(v_h;\cdot)-A_h(w_h;\cdot)\|_{V_h'}
\le a_1\bigl(1+C\alpha K^2\bigr)\|v_h-w_h\|_V \label{eq:model-local-lipschitz-continuity}
\end{equation}
where \(C\) depends only on the dimension and shape regularity of the finite element mesh. The derivative is given by \eqref{eq:model-derivative}--\eqref{eq:model-derivative-quartic-part},
\begin{align}
A_h'(u_h)[z_h,r_h]
&=\int_{\Omega}
a_{\varepsilon}(x)
\bigl(1+\alpha|\nabla u_h|^2\bigr)
\nabla z_h\cdot\nabla r_h\,dx \label{eq:model-derivative}\\
&\quad+2\alpha\int_{\Omega}
a_{\varepsilon}(x)
(\nabla u_h\cdot\nabla z_h)(\nabla u_h\cdot\nabla r_h)\,dx \label{eq:model-derivative-quartic-part}
\end{align}
and it satisfies, for all \(u_h,z_h\in V_h\),
\begin{equation}
A_h'(u_h)[z_h,z_h]
\ge a_0\|z_h\|_V^2 \label{eq:model-tangent-coercivity}
\end{equation}
\end{prop}
\begin{proof}
Set
\begin{equation}
g(\boldsymbol\xi)=\bigl(1+\alpha|\boldsymbol\xi|^2\bigr)\boldsymbol\xi \label{eq:model-flux-map}
\end{equation}
Then, for all \(\boldsymbol\xi,\boldsymbol z\in\mathbb R^d\),
\begin{equation}
Dg(\boldsymbol\xi)\boldsymbol z\cdot\boldsymbol z
=\bigl(1+\alpha|\boldsymbol\xi|^2\bigr)|\boldsymbol z|^2
+2\alpha(\boldsymbol\xi\cdot\boldsymbol z)^2
\ge |\boldsymbol z|^2 \label{eq:model-flux-jacobian-coercivity}
\end{equation}
Hence, for all \(\boldsymbol\xi,\boldsymbol\eta\in\mathbb R^d\),
\begin{align}
\bigl(g(\boldsymbol\xi)-g(\boldsymbol\eta)\bigr)\cdot(\boldsymbol\xi-\boldsymbol\eta)
&=\int_0^1 Dg\bigl(\boldsymbol\eta+s(\boldsymbol\xi-\boldsymbol\eta)\bigr)(\boldsymbol\xi-\boldsymbol\eta)\cdot(\boldsymbol\xi-\boldsymbol\eta)\,ds \label{eq:model-pointwise-monotonicity-identity}\\
&\ge |\boldsymbol\xi-\boldsymbol\eta|^2 \label{eq:model-pointwise-monotonicity-bound}
\end{align}
After multiplication by \(a_{\varepsilon}\) and integration over \(\Omega\), this gives \eqref{eq:model-strong-monotonicity}. The pointwise estimate
\begin{equation}
|g(\boldsymbol\xi)-g(\boldsymbol\eta)|
\le \bigl(1+C\alpha(|\boldsymbol\xi|^2+|\boldsymbol\eta|^2)\bigr)|\boldsymbol\xi-\boldsymbol\eta|
\quad \forall \boldsymbol\xi,\boldsymbol\eta\in\mathbb R^d \label{eq:model-pointwise-lipschitz}
\end{equation}
gives \eqref{eq:model-local-lipschitz-continuity} on \(\mathcal B_K\). Differentiating \eqref{eq:model-nonlinear-form} with respect to the first argument gives \eqref{eq:model-derivative}--\eqref{eq:model-derivative-quartic-part}, and \eqref{eq:model-tangent-coercivity} follows from \eqref{eq:model-flux-jacobian-coercivity}.
\end{proof}

\paragraph{Model Patch Problems.}
\label{par:model-patch-problems}
For \(v_H\in V_H\) and \(i\in\mathcal I\), the localized correction \(v_{f,i}^{\ell}(P_i v_H)\in V_{f,i}^{\ell}\) from Section \ref{sec:localized-reconstruction} is obtained by the same constitutive law on the patch. Define
\begin{equation}
u_i^{\ell}=v_H+v_{f,i}^{\ell}(P_i v_H) \label{eq:model-local-state-definition}
\end{equation}
Then the nonlinear patch problem is
\begin{align}
\int_{\omega_i^{\ell}}
a_{\varepsilon}(x)
\bigl(1+\alpha|\nabla u_i^{\ell}|^2\bigr)
\nabla u_i^{\ell}\cdot\nabla w_{f,i}\,dx
&=\int_{\omega_i^{\ell}}fw_{f,i}\,dx
\quad \forall w_{f,i}\in V_{f,i}^{\ell} \label{eq:model-local-patch-problem}
\end{align}
Only the restriction of \(v_H\) to \(\omega_i^{\ell}\), represented by \(P_i v_H\), enters \eqref{eq:model-local-patch-problem}. For a direction \(\delta v_H\in V_H\), the patch tangent bilinear form is given by \eqref{eq:model-local-tangent-form-principal}--\eqref{eq:model-local-tangent-form-quartic},
\begin{align}
B_i^{\ell}(z_h,w_{f,i})
&=\int_{\omega_i^{\ell}}
a_{\varepsilon}(x)
\bigl(1+\alpha|\nabla u_i^{\ell}|^2\bigr)
\nabla z_h\cdot\nabla w_{f,i}\,dx \label{eq:model-local-tangent-form-principal}\\
&\quad+2\alpha\int_{\omega_i^{\ell}}
a_{\varepsilon}(x)
(\nabla u_i^{\ell}\cdot\nabla z_h)
(\nabla u_i^{\ell}\cdot\nabla w_{f,i})\,dx \label{eq:model-local-tangent-form-quartic}
\end{align}
The tangent correction \(q_{f,i}^{\ell}\in V_{f,i}^{\ell}\) in the direction \(\delta v_H\) is defined by
\begin{equation}
B_i^{\ell}(\delta v_H+q_{f,i}^{\ell},w_{f,i})=0
\quad \forall w_{f,i}\in V_{f,i}^{\ell} \label{eq:model-local-tangent-patch-problem}
\end{equation}
and uniqueness gives the identity
\begin{equation}
q_{f,i}^{\ell}=D v_{f,i}^{\ell}(P_i v_H)[P_i\delta v_H] \label{eq:model-local-tangent-correction-identity}
\end{equation}
The coercivity bound \eqref{eq:model-tangent-coercivity} holds on each patch with the same lower constant \(a_0\), since functions in \(V_{f,i}^{\ell}\) satisfy homogeneous patch boundary conditions. Thus the tangent patch problems are uniformly solvable.

\paragraph{Finite-Dimensional Consequences.}
\label{par:model-finite-dimensional-consequences}
After choosing a basis of \(V_{f,i}^{\ell}\), \eqref{eq:model-local-patch-problem} is a finite-dimensional nonlinear algebraic system with parameter \(P_i v_H\). The Jacobian of this system with respect to the fine-scale unknown is the tangent patch operator in \eqref{eq:model-local-tangent-patch-problem}. Strong monotonicity gives uniqueness of the nonlinear patch correction, tangent coercivity gives invertibility of the tangent patch operator, and the implicit-function theorem gives differentiability of the local patch map on admissible parameter sets. The restricted maps \(\mathcal G_{f,i}^{\ell}=R_i v_{f,i}^{\ell}\) are therefore well-defined \(C^1\) targets for the optional network interpolation in Section \ref{sec:network-interpolation}.

%% file: sections/error_analysis_v5.tex

\section{Stability and Error Decomposition}
\label{sec:error-analysis}

This section gives a discrete perturbation framework for the localized manifold method. All estimates are relative to the fine-scale Galerkin solution \(u_h\). The finite element discretization error between \(u_h\) and the exact weak solution is not included. The main point is to separate the geometric stability mechanism from the localization and residual defects. In particular, the chord error generated by the curvature of the reconstructed manifold is absorbed into the monotonicity constant, while residual defects from inexact nonlinear solves or learned surrogates remain on the right-hand side.


\subsection{Reference Problem and Structural Assumptions}
\label{subsec:reference-problem-and-structural-assumptions}

\paragraph{Residual Notation.}
\label{par:error-residual-notation}
For \(z_h\in V_h\), define the residual functional by
\begin{equation}
\operatorname{Res}(z_h)(w_h)=L(w_h)-A(z_h;w_h)
\quad \forall w_h\in V_h \label{eq:error-residual-functional}
\end{equation}
The fine-scale solution satisfies \(\operatorname{Res}(u_h)(w_h)=0\) for all \(w_h\in V_h\).

\paragraph{Admissible Sets.}
\label{par:error-admissible-sets}
The argument is local in the nonlinear state. We assume that the relevant fine-scale states belong to an admissible set \(\mathcal U_h\subset V_h\), and that the relevant coarse states belong to an admissible set \(\mathcal U_H\subset V_H\). On \(\mathcal U_h\), the fine-scale operator is strongly monotone and locally Lipschitz continuous,
\begin{equation}
m_A\|v_h-w_h\|_V^2\le A(v_h;v_h-w_h)-A(w_h;v_h-w_h)
\quad \forall v_h,w_h\in\mathcal U_h \label{eq:error-strong-monotonicity-assumption}
\end{equation}
and
\begin{equation}
\|A(v_h;\cdot)-A(w_h;\cdot)\|_{V_h'}
\le L_A\|v_h-w_h\|_V
\quad \forall v_h,w_h\in\mathcal U_h \label{eq:error-lipschitz-continuity-assumption}
\end{equation}
with constants \(m_A>0\) and \(L_A>0\). Section \ref{sec:heterogeneous-nonlinear-diffusion} verifies these assumptions for the heterogeneous monotone diffusion model.

\paragraph{Projection and Section Properties.}
\label{par:error-projection-and-section-properties}
The ideal global reconstruction \(M_h\) and the computable reconstructions used below are sections of the coarse projection. Thus,
\begin{align}
\Pi_H M_h(v_H)&=v_H
\quad \forall v_H\in\mathcal U_H \label{eq:error-ideal-section-property}\\
\Pi_H\widetilde M_h^\ell(v_H)&=v_H
\quad \forall v_H\in\mathcal U_H \label{eq:error-computable-section-property}
\end{align}
The first identity implies the exact representation of the fine-scale Galerkin solution,
\begin{equation}
u_h=M_h(\Pi_Hu_h) \label{eq:error-exact-global-manifold-representation}
\end{equation}
We also assume that \(\Pi_H\) is stable in the \(V\)-norm,
\begin{equation}
\|\Pi_Hv_h\|_V\le C_\Pi\|v_h\|_V
\quad \forall v_h\in V_h \label{eq:error-projection-stability}
\end{equation}

\paragraph{Generic Reconstruction.}
\label{par:error-generic-reconstruction}
Let \(\widetilde M_h^\ell\) denote either the patch-solved reconstruction or the network-interpolated reconstruction,
\begin{equation}
\widetilde M_h^\ell=M_h^\ell
\quad \text{or} \quad
\widetilde M_h^\ell=M_{h,\theta}^\ell \label{eq:error-generic-reconstruction}
\end{equation}
and let \(\widetilde u_H^\ell\in V_H\) be the computed coarse state. We set
\begin{equation}
\widetilde u_{\mathrm{ms}}^\ell=\widetilde M_h^\ell(\widetilde u_H^\ell) \label{eq:error-generic-ms-solution}
\end{equation}
and compare it with the same reconstruction evaluated at the exact coarse coordinate,
\begin{equation}
\overline u_h^\ell=\widetilde M_h^\ell(\Pi_Hu_h) \label{eq:error-comparison-state}
\end{equation}
The coarse displacement and the coarse segment used below are
\begin{align}
\delta u_H^\ell&=\Pi_Hu_h-\widetilde u_H^\ell \label{eq:error-coarse-difference}\\
v_H(s)&=\widetilde u_H^\ell+s\delta u_H^\ell
\quad 0\le s\le 1 \label{eq:error-coarse-segment}
\end{align}
We assume that \(v_H(s)\in\mathcal U_H\) for \(0\le s\le1\), and that \(u_h\), \(\widetilde u_{\mathrm{ms}}^\ell\), and \(\overline u_h^\ell\) belong to \(\mathcal U_h\). We also assume that \(\widetilde M_h^\ell\) is differentiable on \(\mathcal U_H\).

\paragraph{Local-to-Global Stability.}
\label{par:error-local-to-global-stability}
The localized and network-interpolated reconstructions assemble restricted patch outputs on \(\omega_i^0\). We use the following stability bounds. For every family \((w_i)_{i\in\mathcal I}\), with \(w_i\in W_{f,i}^0\),
\begin{equation}
\biggl\|
(I-\Pi_H)I_h
\biggl(
\sum_{i\in\mathcal I}\varphi_{H,i}E_{i,0}^{\operatorname{ext}}w_i
\biggr)
\biggr\|_V
\le C_{\operatorname{pu}}
\biggl(\sum_{i\in\mathcal I}\|w_i\|_{H^1(\omega_i^0)}^2\biggr)^{1/2} \label{eq:partition-of-unity-stability}
\end{equation}
Moreover, for each local coarse restriction \(P_i\), there is a constant \(C_{P,i}\) such that
\begin{equation}
\|P_i\delta v_H\|_{\mathbb R^{n_{H,i}}}
\le C_{P,i}\|\delta v_H\|_V
\quad \forall \delta v_H\in V_H \label{eq:local-coarse-restriction-stability}
\end{equation}
The Euclidean norm in \eqref{eq:local-coarse-restriction-stability} is understood after the local coordinate scaling used for the patch problems and training sets. With this convention, the constants are stable under quasi-uniform refinement.


\subsection{Geometric Error Split}
\label{subsec:geometric-error-split}

The perturbation estimate is obtained by testing the residual at the computed multiscale state with the true error. The key point is that this error can be decomposed into three geometrically distinct contributions. First, we compare the ideal and localized reconstructions at the exact coarse coordinate. Second, we measure the difference between the chord on the localized manifold and its tangent approximation at the computed coarse state. Third, we isolate the tangent vector generated by the coarse-scale error, which is an admissible test function in the coarse problem,
\begin{align}
u_h-\widetilde u_{\mathrm{ms}}^{\ell}
&= M_h(\Pi_Hu_h)- \widetilde M_h^{\ell}(\widetilde u_H^\ell) 
\label{eq:error-direct-split-total}\\
&=M_h(\Pi_Hu_h) - \widetilde M_h^{\ell}(\Pi_Hu_h)
+
\widetilde M_h^{\ell}(\Pi_Hu_h) - \widetilde M_h^{\ell}(\widetilde u_H^\ell) 
\label{eq:error-direct-split-add-reconstruction}\\
&=M_h(\Pi_Hu_h) - \widetilde M_h^{\ell}(\Pi_Hu_h) \label{eq:error-direct-split-state}\\
&\quad+
\widetilde M_h^{\ell}(\Pi_Hu_h)
-
\Big( \widetilde M_h^{\ell}(\widetilde u_H^\ell) 
+ D\widetilde M_h^{\ell}(\widetilde u_H^{\ell})
[\Pi_Hu_h-\widetilde u_H^{\ell}] \Big) \label{eq:error-direct-split-chord-tangent}\\
&\quad+
D\widetilde M_h^{\ell}(\widetilde u_H^{\ell})
[\Pi_Hu_h-\widetilde u_H^{\ell}] \label{eq:error-direct-split-tangent}
\end{align}
We denote the three contributions in the final decomposition by
\begin{align}
e_{\mathrm{loc}}^\ell
&\coloneqq
M_h(\Pi_Hu_h)
-
\widetilde M_h^{\ell}(\Pi_Hu_h)
\label{eq:error-localization-vector}\\
e_{\mathrm{ch}}^\ell
&\coloneqq
\widetilde M_h^{\ell}(\Pi_Hu_h)
-
\Big( \widetilde M_h^{\ell}(\widetilde u_H^\ell)
+
D\widetilde M_h^{\ell}(\widetilde u_H^{\ell})
[\Pi_Hu_h-\widetilde u_H^{\ell}] \Big)
\label{eq:error-chord-vector}\\
z_{\mathrm{tan}}^\ell
&\coloneqq
D\widetilde M_h^{\ell}(\widetilde u_H^{\ell})
[\Pi_Hu_h-\widetilde u_H^{\ell}]
\label{eq:error-tangent-test-vector}
\end{align}
Thus,
\begin{equation}
u_h-\widetilde u_{\mathrm{ms}}^{\ell}
=
e_{\mathrm{loc}}^\ell
+
e_{\mathrm{ch}}^\ell
+
z_{\mathrm{tan}}^\ell
\label{eq:error-split-compact}
\end{equation}

The term \(e_{\mathrm{loc}}^\ell\) is the localization error, evaluated at the exact coarse coordinate \(\Pi_Hu_h\). The term \(e_{\mathrm{ch}}^\ell\) is the chord error, measuring the nonlinear defect between the localized chord and the tangent approximation at \(\widetilde u_H^\ell\). Finally, \(z_{\mathrm{tan}}^\ell\) is the image of the coarse-scale error \(\Pi_Hu_h-\widetilde u_H^\ell\) under the localized tangent map. In particular, \(z_{\mathrm{tan}}^\ell\) belongs to the admissible tangent test space used in the coarse problem.

\begin{figure}[htbp]
\centering
\includegraphics[width=0.92\textwidth]{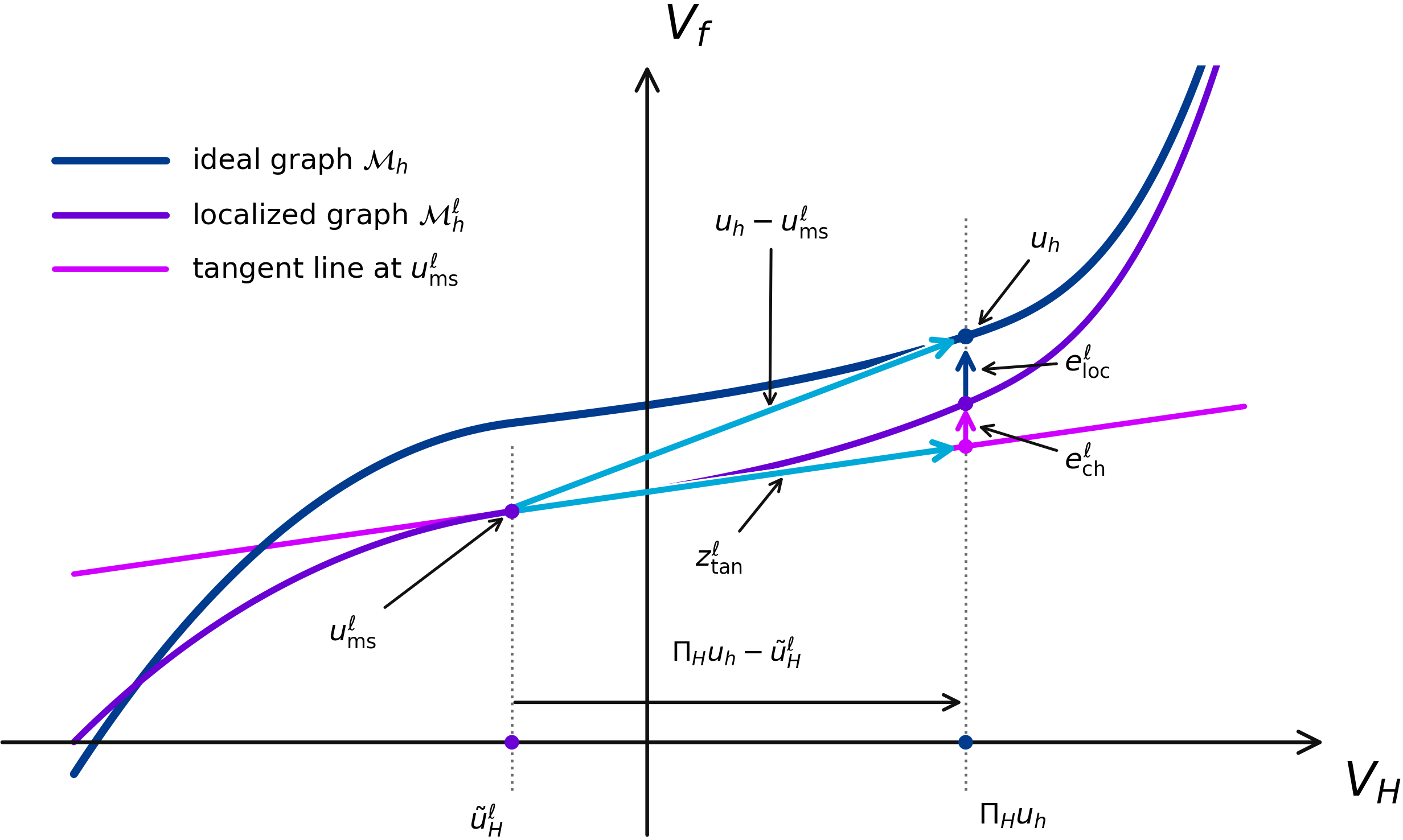}
\caption{Geometric interpretation of the error representation. The total error from \(\widetilde u_{\mathrm{ms}}^\ell\) to \(u_h\) is written as the tangent image \(z_{\mathrm{tan}}^\ell\), followed by the chord error \(e_{\mathrm{ch}}^\ell\), followed by the localization error \(e_{\mathrm{loc}}^\ell\).}
\label{fig:error-splitting-decomposition}
\end{figure}

\paragraph{Coarse Projection of the Graph Error.}
\label{par:coarse-projection-of-the-graph-error}
The estimates close because the coarse displacement is the coarse projection of the full graph error. Set
\begin{equation}
e_h\coloneqq u_h-\widetilde u_{\mathrm{ms}}^\ell \label{eq:error-total-error-vector}
\end{equation}
Using \eqref{eq:error-ideal-section-property} and \eqref{eq:error-computable-section-property}, we obtain
\begin{align}
\delta u_H^\ell
&=\Pi_Hu_h-\Pi_H\widetilde u_{\mathrm{ms}}^\ell \label{eq:error-coarse-error-projection-identity}\\
&=\Pi_H(u_h-\widetilde u_{\mathrm{ms}}^\ell) \label{eq:error-coarse-error-projected-total-error}\\
&=\Pi_He_h \label{eq:error-coarse-error-projected-error-vector}
\end{align}
Therefore, by \eqref{eq:error-projection-stability},
\begin{equation}
\|\delta u_H^\ell\|_V\le C_\Pi\|e_h\|_V \label{eq:error-coarse-error-projection-bound}
\end{equation}
This observation is used twice below. Quadratic dependence on \(\delta u_H^\ell\), as in the chord error, is absorbed into the monotonicity constant. Linear dependence on \(\delta u_H^\ell\), as in a reduced residual or a tangent-surrogate defect, remains as a right-hand-side perturbation.


\subsection{Chord Curvature and Kickback}
\label{subsec:chord-curvature-and-kickback}

The chord error is a curvature term for the reconstruction actually used in the coarse problem. It is not a separate localization defect. Under a local Lipschitz bound for the tangent reconstruction, the chord error is quadratic in the coarse displacement, so the residual it contributes is cubic in the multiscale error and can be kicked back into the monotonicity estimate once that error is small.

\begin{lem}[Chord curvature kickback]
\label{lem:chord-curvature-kickback}
Assume that \(D\widetilde M_h^\ell\) is Lipschitz on \(\mathcal U_H\), in the sense that
\begin{align}
\bigl\|
\bigl(
D\widetilde M_h^\ell(v_H)
-
D\widetilde M_h^\ell(w_H)
\bigr)[\xi_H]
\bigr\|_V
&\le
L_{\widetilde M,1}^\ell
\|v_H-w_H\|_V
\|\xi_H\|_V
\label{eq:error-localized-reconstruction-derivative-lipschitz}\\
&\quad \forall v_H,w_H\in\mathcal U_H
\quad \forall \xi_H\in V_H
\notag
\end{align}
Then the chord residual satisfies
\begin{equation}
\bigl|
\operatorname{Res}(\widetilde u_{\mathrm{ms}}^\ell)(e_{\mathrm{ch}}^\ell)
\bigr|
\le
\gamma_{\mathrm{ch}}^\ell
\|e_h\|_V^3 \label{eq:error-chord-residual-kickback-bound}
\end{equation}
with
\begin{equation}
\gamma_{\mathrm{ch}}^\ell
\coloneqq
\frac12
L_A L_{\widetilde M,1}^\ell C_\Pi^2 \label{eq:error-chord-kickback-constant}
\end{equation}
\end{lem}

\begin{proof}
By the definition of \(e_{\mathrm{ch}}^\ell\) and the mean-value identity along the segment \eqref{eq:error-coarse-segment},
\begin{align}
e_{\mathrm{ch}}^\ell
&=
\widetilde M_h^\ell(\widetilde u_H^\ell+\delta u_H^\ell)
-
\widetilde M_h^\ell(\widetilde u_H^\ell)
-
D\widetilde M_h^\ell(\widetilde u_H^\ell)[\delta u_H^\ell]
\label{eq:error-chord-error-expanded}\\
&=
\int_0^1
\bigl(
D\widetilde M_h^\ell(v_H(s))
-
D\widetilde M_h^\ell(\widetilde u_H^\ell)
\bigr)
[\delta u_H^\ell] \,ds
\label{eq:error-chord-error-integral-identity}
\end{align}
Using \eqref{eq:error-localized-reconstruction-derivative-lipschitz}, we get
\begin{align}
\|e_{\mathrm{ch}}^\ell\|_V
&\le
\int_0^1
L_{\widetilde M,1}^\ell
\|v_H(s)-\widetilde u_H^\ell\|_V
\|\delta u_H^\ell\|_V \,ds
\label{eq:error-chord-bound-integral}\\
&=
\int_0^1
sL_{\widetilde M,1}^\ell
\|\delta u_H^\ell\|_V^2 \,ds
\label{eq:error-chord-bound-segment}\\
&=
\frac12
L_{\widetilde M,1}^\ell
\|\delta u_H^\ell\|_V^2
\label{eq:error-chord-bound-quadratic}
\end{align}
Since \(\operatorname{Res}(u_h)=0\), the Lipschitz continuity \eqref{eq:error-lipschitz-continuity-assumption} bounds the residual at the computed state by the error,
\begin{equation}
\|\operatorname{Res}(\widetilde u_{\mathrm{ms}}^\ell)\|_{V_h'}
=\|A(u_h;\cdot)-A(\widetilde u_{\mathrm{ms}}^\ell;\cdot)\|_{V_h'}
\le L_A\|e_h\|_V
\label{eq:error-residual-at-computed-state}
\end{equation}
Combining this with the chord bound \eqref{eq:error-chord-bound-quadratic} and the projection bound \eqref{eq:error-coarse-error-projection-bound} gives
\begin{align}
\bigl|
\operatorname{Res}(\widetilde u_{\mathrm{ms}}^\ell)(e_{\mathrm{ch}}^\ell)
\bigr|
&\le
\|\operatorname{Res}(\widetilde u_{\mathrm{ms}}^\ell)\|_{V_h'}\|e_{\mathrm{ch}}^\ell\|_V
\label{eq:error-chord-residual-scale-bound}\\
&\le
\frac12
L_A L_{\widetilde M,1}^\ell
\|e_h\|_V\|\delta u_H^\ell\|_V^2
\label{eq:error-chord-residual-quadratic-coarse}\\
&\le
\frac12
L_A L_{\widetilde M,1}^\ell C_\Pi^2
\|e_h\|_V^3
\label{eq:error-chord-residual-quadratic-total}
\end{align}
This is \eqref{eq:error-chord-residual-kickback-bound} with \(\gamma_{\mathrm{ch}}^\ell\) defined by \eqref{eq:error-chord-kickback-constant}.
\end{proof}

\begin{figure}[htbp]
\centering
\includegraphics[width=0.92\textwidth]{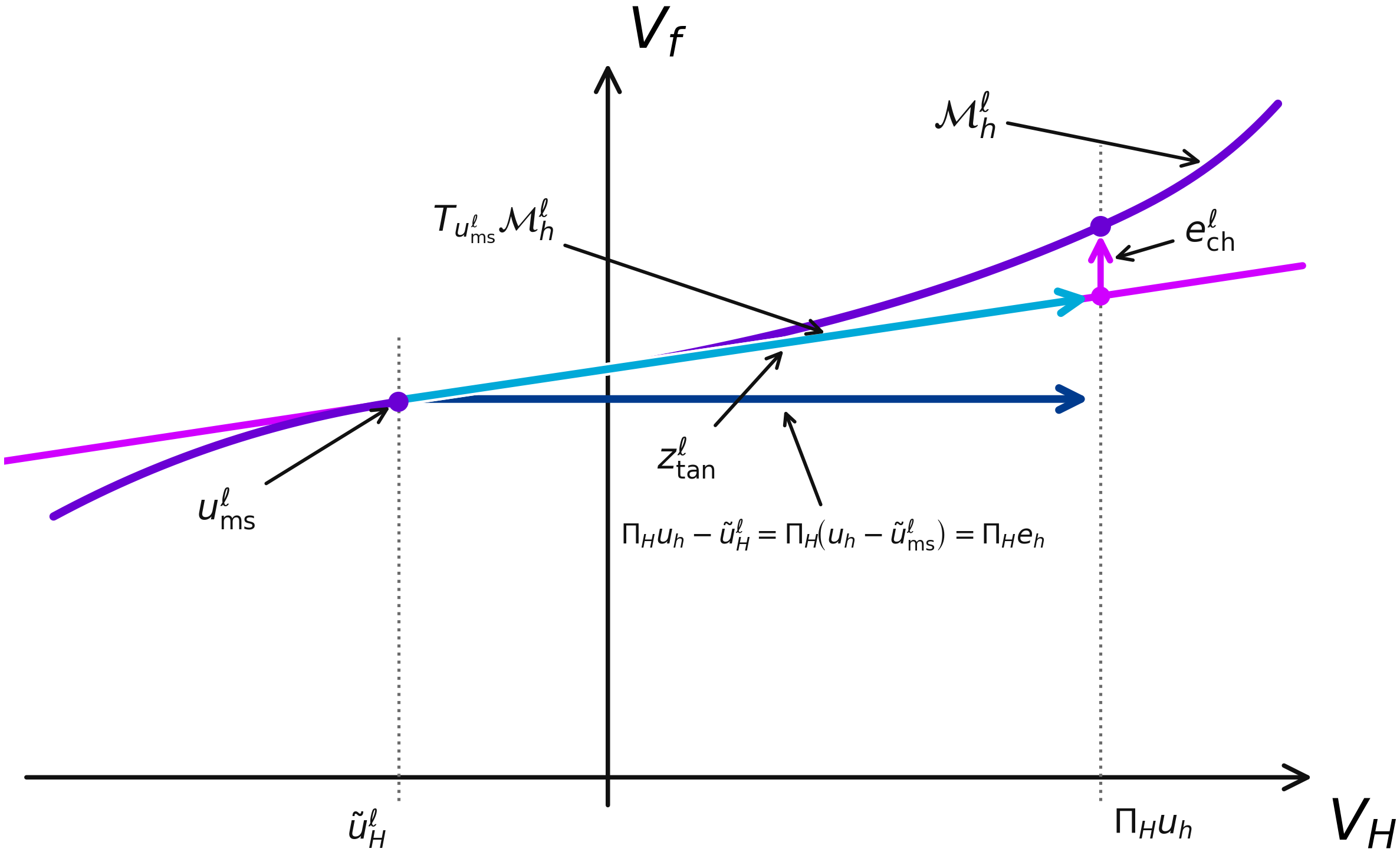}
\caption{Geometric interpretation of the chord error estimate. The tangent correction \(z_{\mathrm{tan}}^\ell\) transports the localized multiscale solution \(u_{\mathrm{ms}}^\ell\) along \(T_{u_{\mathrm{ms}}^\ell}\mathcal M_h^\ell\). The chord error \(e_{\mathrm{ch}}^\ell\) is the remaining vertical discrepancy between this tangent approximation and the nonlinear manifold \(\mathcal M_h^\ell\) at the coarse coordinate \(\Pi_H u_h\), where
\(\Pi_H u_h-\widetilde u_H^\ell=\Pi_H(u_h-\widetilde u_{\mathrm{ms}}^\ell)=\Pi_H e_h\).}
\label{fig:tangent-chord-projection}
\end{figure}


\subsection{Exact Localized Coarse Solve}
\label{subsec:exact-localized-coarse-solve}

We first consider the clean case in which the chosen localized or learned coarse problem is solved exactly. In this case the tangent term \(z_{\mathrm{tan}}^\ell\) in \eqref{eq:error-split-compact} makes no contribution to the residual. This is not because \(z_{\mathrm{tan}}^\ell\) vanishes—in general it does not—but because the residual tested against it vanishes, \(\operatorname{Res}(\widetilde u_{\mathrm{ms}}^\ell)(z_{\mathrm{tan}}^\ell)=0\): the vector \(z_{\mathrm{tan}}^\ell\) is an admissible tangent test function, and the exactly solved coarse equation \eqref{eq:error-exact-coarse-problem} forces the residual to vanish on the entire tangent test space (the nonlinear analogue of Galerkin orthogonality). The only remaining right-hand-side defect is then the localization or reconstruction error \(e_{\mathrm{loc}}^\ell\).

\begin{prop}[Exact localized coarse solve]
\label{prop:exact-localized-coarse-solve}
Let \(\widetilde u_{\mathrm{ms}}^\ell=\widetilde M_h^\ell(\widetilde u_H^\ell)\) and assume \eqref{eq:error-strong-monotonicity-assumption}, \eqref{eq:error-lipschitz-continuity-assumption}, and the hypotheses of Lemma \ref{lem:chord-curvature-kickback}. Assume further that the coarse state solves the tangent-space coarse problem exactly,
\begin{equation}
\operatorname{Res}(\widetilde u_{\mathrm{ms}}^\ell)
\bigl(
D\widetilde M_h^\ell(\widetilde u_H^\ell)[\xi_H]
\bigr)
=0
\quad \forall \xi_H\in V_H \label{eq:error-exact-coarse-problem}
\end{equation}
If the error is sufficiently small that \(\gamma_{\mathrm{ch}}^\ell\|e_h\|_V\le\tfrac12 m_A\), then
\begin{equation}
\|u_h-\widetilde u_{\mathrm{ms}}^\ell\|_V
\le
\frac{2L_A}{m_A}
\|e_{\mathrm{loc}}^\ell\|_V \label{eq:error-exact-coarse-solve-bound}
\end{equation}
\end{prop}

\begin{proof}
Set \(e_h=u_h-\widetilde u_{\mathrm{ms}}^\ell\). Strong monotonicity and the fine-scale equation give
\begin{align}
m_A\|e_h\|_V^2
&\le
A(u_h;e_h)-A(\widetilde u_{\mathrm{ms}}^\ell;e_h)
\label{eq:error-exact-proof-monotonicity}\\
&=
L(e_h)-A(\widetilde u_{\mathrm{ms}}^\ell;e_h)
\label{eq:error-exact-proof-fine-scale-equation}\\
&=
\operatorname{Res}(\widetilde u_{\mathrm{ms}}^\ell)(e_h)
\label{eq:error-exact-proof-residual-error}
\end{align}
Using \eqref{eq:error-split-compact}, the triangle inequality, and \eqref{eq:error-exact-coarse-problem} with \(\xi_H=\delta u_H^\ell\), we obtain
\begin{align}
m_A\|e_h\|_V^2
&\le
\bigl|
\operatorname{Res}(\widetilde u_{\mathrm{ms}}^\ell)(e_{\mathrm{loc}}^\ell)
\bigr|
+
\bigl|
\operatorname{Res}(\widetilde u_{\mathrm{ms}}^\ell)(e_{\mathrm{ch}}^\ell)
\bigr|
\label{eq:error-exact-proof-split-bound}
\end{align}
The localization term is controlled by Lipschitz continuity. Since \(u_h\) solves the fine-scale problem,
\begin{align}
\bigl|
\operatorname{Res}(\widetilde u_{\mathrm{ms}}^\ell)(e_{\mathrm{loc}}^\ell)
\bigr|
&=
\bigl|
A(u_h;e_{\mathrm{loc}}^\ell)
-
A(\widetilde u_{\mathrm{ms}}^\ell;e_{\mathrm{loc}}^\ell)
\bigr|
\label{eq:error-exact-proof-localization-identity}\\
&\le
L_A\|e_h\|_V\|e_{\mathrm{loc}}^\ell\|_V
\label{eq:error-exact-proof-localization-bound}
\end{align}
The chord term is bounded by Lemma \ref{lem:chord-curvature-kickback}. Hence,
\begin{equation}
m_A\|e_h\|_V^2
\le
L_A\|e_{\mathrm{loc}}^\ell\|_V\|e_h\|_V
+
\gamma_{\mathrm{ch}}^\ell\|e_h\|_V^3 \label{eq:error-exact-proof-before-kickback}
\end{equation}
Under the smallness assumption \(\gamma_{\mathrm{ch}}^\ell\|e_h\|_V\le\tfrac12 m_A\), the chord term is at most \(\tfrac12 m_A\|e_h\|_V^2\). Moving it to the left-hand side gives
\begin{equation}
\tfrac12 m_A\|e_h\|_V^2
\le
L_A\|e_{\mathrm{loc}}^\ell\|_V\|e_h\|_V \label{eq:error-exact-proof-after-kickback}
\end{equation}
If \(e_h\neq0\), division by \(\|e_h\|_V\) gives \eqref{eq:error-exact-coarse-solve-bound}. If \(e_h=0\), the estimate is immediate.
\end{proof}


\subsection{Deterministic Localization Error}
\label{subsec:deterministic-localization-error}

For the patch-solved method, the computable map is \(M_h^\ell\). Its deterministic state reconstruction error is measured relative to the ideal global map \(M_h\). On the admissible coarse set \(\mathcal U_H\), define
\begin{equation}
\eta_{\mathrm{loc}}(H,\ell)
\coloneqq
\sup_{v_H\in\mathcal U_H}
\|M_h(v_H)-M_h^\ell(v_H)\|_V \label{eq:deterministic-localization-state-defect}
\end{equation}
Equivalently, one may write \(\eta_{\mathrm{loc}}^0=\eta_{\mathrm{loc}}(H,\ell)\). Since \(u_h=M_h(\Pi_Hu_h)\), the localization vector for the patch-solved method satisfies
\begin{equation}
\|e_{\mathrm{loc}}^\ell\|_V
=
\|M_h(\Pi_Hu_h)-M_h^\ell(\Pi_Hu_h)\|_V
\le
\eta_{\mathrm{loc}}(H,\ell) \label{eq:localized-state-defect-bound}
\end{equation}
Combining \eqref{eq:localized-state-defect-bound} with Proposition \ref{prop:exact-localized-coarse-solve} gives the following deterministic consequence.

\begin{prop}[Patch-solved localization bound]
\label{prop:patch-solved-localization-bound}
Let \(u_{\mathrm{ms}}^\ell=M_h^\ell(u_H^\ell)\) be the patch-solved localized solution, and assume that \(u_H^\ell\) solves the localized tangent-space coarse problem exactly. If the assumptions of Proposition \ref{prop:exact-localized-coarse-solve} hold with \(\widetilde M_h^\ell=M_h^\ell\), then
\begin{equation}
\|u_h-u_{\mathrm{ms}}^\ell\|_V
\le
\frac{2L_A}{m_A}
\eta_{\mathrm{loc}}(H,\ell) \label{eq:patch-solved-localization-bound}
\end{equation}
provided the multiscale error is sufficiently small that \(\gamma_{\mathrm{ch}}^\ell\|u_h-u_{\mathrm{ms}}^\ell\|_V\le\tfrac12 m_A\).
\end{prop}

\begin{rem}[Localization estimate]
\label{rem:localization-estimate}
The present paper uses \(\eta_{\mathrm{loc}}(H,\ell)\) as the deterministic localization quantity. A proof of decay in terms of \(H\), the oversampling parameter \(\ell\), coefficient contrast, and nonlinear stability constants is deferred to a forthcoming localization analysis. Thus Proposition \ref{prop:patch-solved-localization-bound} should be read as a stability transfer result: any estimate for \(\eta_{\mathrm{loc}}(H,\ell)\) immediately yields a corresponding multiscale error estimate.
\end{rem}


\subsection{Residual Defects and Inexact Coarse Solves}
\label{subsec:residual-defects-and-inexact-coarse-solves}

We now allow the tangent-space coarse problem to be solved inexactly. The third term in \eqref{eq:error-split-compact} is then controlled by the reduced residual of the coarse problem.

\paragraph{Reduced Residual.}
\label{par:reduced-residual}
For a reconstruction \(\widetilde M_h^\ell\), define the reduced residual by
\begin{equation}
\mathcal R_H^\ell(w_H)(\xi_H)
\coloneqq
\operatorname{Res}\bigl(\widetilde M_h^\ell(w_H)\bigr)
\bigl(
D\widetilde M_h^\ell(w_H)[\xi_H]
\bigr)
\quad \forall \xi_H\in V_H \label{eq:error-reduced-residual}
\end{equation}
Its dual norm is
\begin{equation}
\|\mathcal R_H^\ell(w_H)\|_{V_H'}
\coloneqq
\sup_{\xi_H\in V_H\setminus\{0\}}
\frac{
|\mathcal R_H^\ell(w_H)(\xi_H)|
}{
\|\xi_H\|_V
} \label{eq:error-reduced-residual-dual-norm}
\end{equation}
By the definition of \(z_{\mathrm{tan}}^\ell\), the tangent residual contribution is
\begin{equation}
\operatorname{Res}(\widetilde u_{\mathrm{ms}}^\ell)(z_{\mathrm{tan}}^\ell)
=
\mathcal R_H^\ell(\widetilde u_H^\ell)(\delta u_H^\ell) \label{eq:error-algebraic-term-as-reduced-residual}
\end{equation}
Consequently, \eqref{eq:error-coarse-error-projection-bound} gives
\begin{align}
\bigl|
\operatorname{Res}(\widetilde u_{\mathrm{ms}}^\ell)(z_{\mathrm{tan}}^\ell)
\bigr|
&\le
\|\mathcal R_H^\ell(\widetilde u_H^\ell)\|_{V_H'}
\|\delta u_H^\ell\|_V
\label{eq:error-reduced-residual-bound-coarse}\\
&\le
C_\Pi
\|\mathcal R_H^\ell(\widetilde u_H^\ell)\|_{V_H'}
\|e_h\|_V
\label{eq:error-reduced-residual-bound-total}
\end{align}

\begin{prop}[Inexact coarse solve]
\label{prop:inexact-coarse-solve}
Let \(\widetilde u_{\mathrm{ms}}^\ell=\widetilde M_h^\ell(\widetilde u_H^\ell)\). Assume \eqref{eq:error-strong-monotonicity-assumption}, \eqref{eq:error-lipschitz-continuity-assumption}, and the hypotheses of Lemma \ref{lem:chord-curvature-kickback}. If the error is sufficiently small that \(\gamma_{\mathrm{ch}}^\ell\|e_h\|_V\le\tfrac12 m_A\), then
\begin{equation}
\|u_h-\widetilde u_{\mathrm{ms}}^\ell\|_V
\le
\frac{2}{m_A}
\bigl(
L_A\|e_{\mathrm{loc}}^\ell\|_V
+
C_\Pi\|\mathcal R_H^\ell(\widetilde u_H^\ell)\|_{V_H'}
\bigr) \label{eq:error-inexact-coarse-solve-bound}
\end{equation}
\end{prop}

\begin{proof}
The proof is identical to the proof of Proposition \ref{prop:exact-localized-coarse-solve}, except that the tangent residual term is not zero. Using \eqref{eq:error-split-compact}, Lemma \ref{lem:chord-curvature-kickback}, \eqref{eq:error-exact-proof-localization-bound}, and \eqref{eq:error-reduced-residual-bound-total}, we obtain
\begin{align}
m_A\|e_h\|_V^2
&\le
L_A\|e_{\mathrm{loc}}^\ell\|_V\|e_h\|_V
+
\gamma_{\mathrm{ch}}^\ell\|e_h\|_V^3
+
C_\Pi\|\mathcal R_H^\ell(\widetilde u_H^\ell)\|_{V_H'}\|e_h\|_V
\label{eq:error-inexact-proof-before-kickback}
\end{align}
Under the smallness assumption \(\gamma_{\mathrm{ch}}^\ell\|e_h\|_V\le\tfrac12 m_A\), the chord term is at most \(\tfrac12 m_A\|e_h\|_V^2\). Moving it to the left-hand side gives
\begin{equation}
\tfrac12 m_A\|e_h\|_V^2
\le
\bigl(
L_A\|e_{\mathrm{loc}}^\ell\|_V
+
C_\Pi\|\mathcal R_H^\ell(\widetilde u_H^\ell)\|_{V_H'}
\bigr)
\|e_h\|_V \label{eq:error-inexact-proof-after-kickback}
\end{equation}
If \(e_h\neq0\), division by \(\|e_h\|_V\) gives \eqref{eq:error-inexact-coarse-solve-bound}. If \(e_h=0\), the estimate is immediate.
\end{proof}

\paragraph{Newton Stopping Criteria.}
\label{par:newton-stopping-criteria}
For Newton-type reduced solves, the dual norm of the reduced residual is the natural stopping quantity. If the iteration is stopped when
\begin{equation}
\|\mathcal R_H^\ell(\widetilde u_H^\ell)\|_{V_H'}
\le
\tau_{\mathrm{alg}} \label{eq:error-newton-stopping-criterion}
\end{equation}
then \eqref{eq:error-inexact-coarse-solve-bound} gives
\begin{equation}
\|u_h-\widetilde u_{\mathrm{ms}}^\ell\|_V
\le
\frac{2}{m_A}
\bigl(
L_A\|e_{\mathrm{loc}}^\ell\|_V
+
C_\Pi\tau_{\mathrm{alg}}
\bigr) \label{eq:error-newton-stopping-error-bound}
\end{equation}
Thus algebraic error enters additively through the reduced residual norm, while the chord curvature is absorbed by monotonicity under the smallness assumption.


\begin{rem}[Regularity of the local patch maps]
\label{rem:patch-map-regularity}
After choosing local bases, the correction problem on a patch \(\omega_i^\ell\) is a finite-dimensional system \(F_i(c_H,c_f)=0\), whose solution defines the implicit map \(c_f=\Psi_i(c_H)\). If \(F_i\) is \(C^r\) and the fine-scale Jacobian \(\partial_{c_f}F_i\) is uniformly invertible along an admissible parameter set---precisely the tangent coercivity that makes the tangent patch solves well posed---then the implicit function theorem yields \(\Psi_i\in C^r\) with
\begin{equation}
D\Psi_i(c_H)
=-\partial_{c_f}F_i(c_H,\Psi_i(c_H))^{-1}\,\partial_{c_H}F_i(c_H,\Psi_i(c_H)) \label{eq:tangent-patch-map-coefficients}
\end{equation}
whose variational form is the tangent patch problem of Sections \ref{sec:localized-reconstruction} and \ref{sec:heterogeneous-nonlinear-diffusion}. The restricted output map \(\mathcal G_{f,i}^\ell\) of Section \ref{sec:network-interpolation} is the composition of \(\Psi_i\) with a linear restriction and therefore inherits the same regularity. This regularity shows that the restricted patch outputs and their tangents are well-defined smooth targets for the network interpolation and for the chord-curvature bound.
\end{rem}


\subsection{Network and Learning Residuals}
\label{subsec:network-and-learning-residuals}

The network-interpolated method is treated as a perturbation of the deterministic localized reconstruction. The present paper only uses realized approximation quantities and residual norms. Explicit a priori estimates for these quantities, including possible Barron-type rates for the restricted patch maps, are deferred to the forthcoming localization and learning analysis.

\paragraph{Realized Local Errors.}
\label{par:realized-local-errors}
For admissible \(v_H\), assume that the restricted learned patch outputs satisfy
\begin{equation}
\|R_i v_{f,i}^\ell(P_i v_H)-g_{f,i}^{\ell,\theta_i}(P_i v_H)\|_{H^1(\omega_i^0)}
\le \varepsilon_i^0 \label{eq:local-network-state-error}
\end{equation}
The partition-of-unity stability estimate \eqref{eq:partition-of-unity-stability} gives the corresponding global state perturbation,
\begin{align}
\|M_h^\ell(v_H)-M_{h,\theta}^\ell(v_H)\|_V
&\le C_{\operatorname{pu}}
\biggl(\sum_{i\in\mathcal I}(\varepsilon_i^0)^2\biggr)^{1/2}
\label{eq:global-network-state-error}\\
\eta_{\mathrm{net}}^0
&\coloneqq
C_{\operatorname{pu}}
\biggl(\sum_{i\in\mathcal I}(\varepsilon_i^0)^2\biggr)^{1/2}
\label{eq:network-state-error-definition}
\end{align}
If derivative information is also used, the analogous active-region tangent error can be measured by
\begin{equation}
\eta_{\mathrm{net}}^1
\coloneqq
C_{\operatorname{pu}}
\biggl(\sum_{i\in\mathcal I}(C_{P,i}\varepsilon_i^1)^2\biggr)^{1/2} \label{eq:network-tangent-error-definition}
\end{equation}
where \(\varepsilon_i^1\) denotes the local derivative error in the norm of \eqref{eq:local-coarse-restriction-stability}. This tangent quantity is not needed as a separate additive term when the coarse problem is defined by \(M_{h,\theta}^\ell\) and the tangent action is computed consistently as \(DM_{h,\theta}^\ell\). It becomes relevant when a learned tangent action is used as an independent surrogate or when one compares learned and deterministic reduced residuals.

\paragraph{Learned Reconstruction Bound.}
\label{par:learned-reconstruction-bound}
When the learned reconstruction \(M_{h,\theta}^\ell\) defines the actual coarse model, the localization vector becomes
\begin{equation}
e_{\mathrm{loc},\theta}^\ell
\coloneqq
M_h(\Pi_Hu_h)-M_{h,\theta}^\ell(\Pi_Hu_h) \label{eq:network-localization-vector}
\end{equation}
By the triangle inequality and \eqref{eq:global-network-state-error},
\begin{align}
\|e_{\mathrm{loc},\theta}^\ell\|_V
&\le
\|M_h(\Pi_Hu_h)-M_h^\ell(\Pi_Hu_h)\|_V
+
\|M_h^\ell(\Pi_Hu_h)-M_{h,\theta}^\ell(\Pi_Hu_h)\|_V
\label{eq:network-localization-vector-triangle}\\
&\le
\eta_{\mathrm{loc}}(H,\ell)+\eta_{\mathrm{net}}^0
\label{eq:network-localization-vector-bound}
\end{align}
Thus Propositions \ref{prop:exact-localized-coarse-solve} and \ref{prop:inexact-coarse-solve} apply with \(\widetilde M_h^\ell=M_{h,\theta}^\ell\). In particular, if the learned coarse problem is solved inexactly and the learned error is sufficiently small that \(\gamma_{\mathrm{ch},\theta}^\ell\|u_h-\widetilde u_{\mathrm{ms},\theta}^\ell\|_V\le\tfrac12 m_A\), then
\begin{equation}
\|u_h-\widetilde u_{\mathrm{ms},\theta}^\ell\|_V
\le
\frac{2}{m_A}
\bigl(
L_A(\eta_{\mathrm{loc}}(H,\ell)+\eta_{\mathrm{net}}^0)
+
C_\Pi\|\mathcal R_{H,\theta}^\ell(\widetilde u_{H,\theta}^\ell)\|_{V_H'}
\bigr) \label{eq:network-inexact-coarse-solve-bound}
\end{equation}
Here \(\mathcal R_{H,\theta}^\ell\) is the reduced residual from \eqref{eq:error-reduced-residual} with \(\widetilde M_h^\ell=M_{h,\theta}^\ell\). There are two complementary ways to account for the learning defect in this estimate. In the bound \eqref{eq:network-inexact-coarse-solve-bound} the learned reconstruction \emph{is} the coarse model: one sets \(\widetilde M_h^\ell=M_{h,\theta}^\ell\), and learning enters \emph{a priori} through the localization term \(\eta_{\mathrm{net}}^0\) of \eqref{eq:network-localization-vector-bound}, through the \emph{learned} reduced residual \(\mathcal R_{H,\theta}^\ell\), and through the \emph{learned} curvature constant \(\gamma_{\mathrm{ch},\theta}^\ell\). Alternatively, a learned state may be viewed as a \emph{candidate} for the deterministic localized problem: one applies Proposition \ref{prop:inexact-coarse-solve} with the deterministic map \(\widetilde M_h^\ell=M_h^\ell\) and its deterministic constants, evaluated at the learned coarse state \(\widetilde u_{H,\theta}^\ell\). Then there is no separate \(\eta_{\mathrm{net}}^0\) term; the entire learning defect is folded into the single \emph{a posteriori}, computable deterministic reduced residual \(\|\mathcal R_H^\ell(\widetilde u_{H,\theta}^\ell)\|_{V_H'}\), which plays exactly the same role as the algebraic residual of an early-stopped outer Newton iteration. The two viewpoints mirror the a priori/a posteriori alternative discussed in Section \ref{subsec:interpretation-for-numerical-experiments}.


\subsection{Interpretation for Numerical Experiments}
\label{subsec:interpretation-for-numerical-experiments}

The estimates identify the quantities that should be measured independently. Oversampling studies probe the deterministic localization quantity \(\eta_{\mathrm{loc}}(H,\ell)\). Nonlinear-solver studies probe the reduced residual norm \(\|\mathcal R_H^\ell(\widetilde u_H^\ell)\|_{V_H'}\). Network studies can either report the same reduced residual as an a posteriori learning indicator, or report active-region state and tangent errors such as \(\eta_{\mathrm{net}}^0\) and \(\eta_{\mathrm{net}}^1\). The chord term is not a separate convergence defect in the final estimate; it enters through the smallness condition \(\gamma_{\mathrm{ch}}^\ell\|e_h\|_V\le\tfrac12 m_A\), which expresses that once the multiscale error is small enough, the local curvature of the reconstructed manifold is absorbed by monotonicity.

%% file: sections/numerical_experiments_v0.tex

\section{Numerical Experiments}
\label{sec:numerical-experiments}

In this section we numerically test the proposed tangent-space multiscale manifold method by applying it to the monotone model problem in Section~\ref{sec:heterogeneous-nonlinear-diffusion} with a heterogeneous coefficient. We first verify the patch-solved method, where the per-patch fine-scale corrections are computed by an inner Newton iteration. We then replace the restricted patch outputs by neural-network interpolants and solve the full learned coarse problem. Finally, we test the same stationary reconstruction in a backward-Euler discretization of a parabolic problem. We consider two strongly heterogeneous coefficients: a smooth oscillatory medium and a non-periodic discontinuous random checkerboard.

\subsection{Model Problem and Discretization}
\label{subsec:numerics-model-problem}

We use the heterogeneous monotone nonlinear diffusion model of Section \ref{sec:heterogeneous-nonlinear-diffusion} on \(\Omega=(0,1)^2\) with homogeneous Dirichlet boundary conditions: the energy \(\mathcal E_h\) \eqref{eq:model-energy}, the nonlinear form \(A=A_h\) \eqref{eq:model-nonlinear-form}, and the load \(L=F_h\) \eqref{eq:model-load-functional}, with \(a_\varepsilon\in L^\infty(\Omega)\), \(0<a_0\le a_\varepsilon\le a_1\), and nonlinearity strength \(\alpha\ge0\). The error is measured in the energy (\(H^1\)) seminorm \(\|\nabla v\|_{\Omega}\).

\paragraph{Discretization.}
\label{par:numerics-discretization}
The domain is triangulated on two nested structured meshes of \(P_1\) Lagrange elements. We fix a single fine reference mesh with \(h=1/64\), that is, \(64\) elements per side, throughout. Most of the experiments use a coarse mesh of \(N=8\) elements per side, so \(H=1/8\). We report the relative energy-norm error of the localized multiscale solution against the fine reference,
\begin{equation}
	\frac{\|\nabla(u_h-u_{\mathrm{ms}}^\ell)\|_{\Omega}}{\|\nabla u_h\|_{\Omega}}
	\label{eq:numerics-relative-error-l-to-ideal}
\end{equation}

The coarse space uses homogeneous Dirichlet data and \(\Pi_H\) is the \(L^2\)-projection realized through the cross-mesh mass constraint. All reported results use vertex patches, one per coarse vertex, with \(\varphi_{H,z}=\Lambda_z\) the coarse hat function. 

\paragraph{Coefficients.}
\label{par:numerics-coefficients}
We report two heterogeneous media, shown in Figure \ref{fig:numerics-coefficients}. The smooth oscillatory coefficient is
\begin{equation}
a_\varepsilon(x,y)=\frac{1}{2}\bigl(2+\beta\sin(2\pi x/\varepsilon)\sin(2\pi y/\varepsilon)\bigr)\quad \varepsilon=\frac{1}{8},\ \beta=1.8
\label{eq:numerics-sine-coefficient}
\end{equation}
Thus \(a_\varepsilon\in[0.1,1.9]\), a contrast ratio of \(19\). The second medium is a non-periodic random checkerboard: the unit square is partitioned into cells of size \(1/32\), and on each cell \(a_\varepsilon\) is drawn independently and uniformly from \([1,20]\) with a fixed seed. This gives a piecewise-constant field with contrast ratio up to \(20\) and no periodic structure. The fine mesh \(h=1/64\) resolves each checkerboard cell with \(2\times2\) elements. Throughout, \(f\equiv1\) and \(\alpha=1\). The non-periodicity of the checkerboard is a deliberate stress on the method, which makes no periodicity assumption.

\begin{figure}[htbp]
\centering
\includegraphics[width=0.95\textwidth]{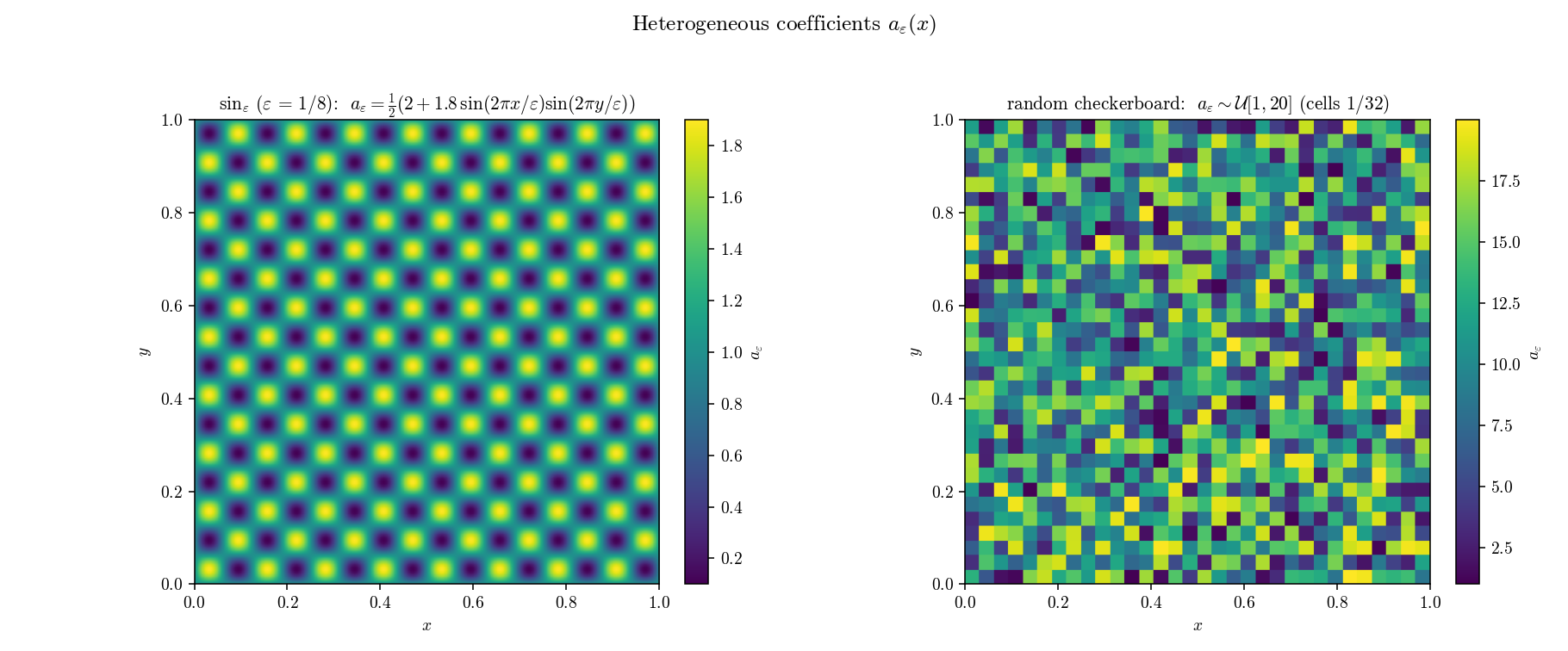}
\caption{The two heterogeneous coefficients \(a_\varepsilon(x)\). Left: the smooth \(\sin_\varepsilon\) medium with \(\varepsilon=1/8\) and \(a_\varepsilon\in[0.1,1.9]\). Right: the non-periodic random checkerboard with cells of size \(1/32\) and \(a_\varepsilon\sim\mathcal U[1,20]\).}
\label{fig:numerics-coefficients}
\end{figure}

\paragraph{The Inexact Galerkin Outer Jacobian.}
\label{par:numerics-inexact-galerkin-outer-jacobian}
Both the patch-solved and the network-interpolated coarse problems are solved by the same outer Newton iteration, with the coarse manifold residual \(r_H(u_H)[\delta v_H]=A(\widetilde M_h^\ell(u_H);D \widetilde M_h^\ell(u_H)[\delta v_H])-L(D \widetilde M_h^\ell(u_H)[\delta v_H])\). As noted in Section \ref{sec:computable-coarse-problems}, its exact derivative in a direction \(w_H\) involves the second derivative of the reconstruction; it splits into two contributions,
\begin{align}
\mathfrak J_H(u_H)[w_H,\delta v_H]
&=\mathfrak J_{H,\operatorname{gal}}^\ell(u_H)[w_H,\delta v_H]
  +\mathfrak J_{H,\operatorname{curv}}^\ell(u_H)[w_H,\delta v_H]
\label{eq:numerics-exact-outer-jacobian-split}\\
\mathfrak J_{H,\operatorname{gal}}^\ell(u_H)[w_H,\delta v_H]
&=A'\bigl(\widetilde M_h^\ell(u_H)\bigr)\bigl[D \widetilde M_h^\ell(u_H)[w_H],D \widetilde M_h^\ell(u_H)[\delta v_H]\bigr]
\label{eq:numerics-exact-outer-jacobian-galerkin}\\
\mathfrak J_{H,\operatorname{curv}}^\ell(u_H)[w_H,\delta v_H]
&=\bigl\langle\mathcal R(\widetilde M_h^\ell(u_H)),D^2\widetilde M_h^\ell(u_H)[w_H,\delta v_H]\bigr\rangle
\label{eq:numerics-exact-outer-jacobian-curvature}
\end{align}
where \(\mathcal R(u_h)=A(u_h;\cdot)-L(\cdot)\) is the fine-scale residual. The Galerkin Jacobian keeps only \eqref{eq:numerics-exact-outer-jacobian-galerkin}, so the outer iteration is an inexact Newton method of Gauss--Newton type. The motivation is threefold. First, \(J_H\) is symmetric positive definite whenever \(A'\) is. Second, the dropped term is controlled by duality,
\begin{equation}
\bigl|\mathfrak J_{H,\operatorname{curv}}^\ell(u_H)[w_H,\delta v_H]\bigr|
\le\bigl\|\mathcal R(\widetilde M_h^\ell(u_H))\bigr\|_{V_h'}\,
\bigl\|D^2\widetilde M_h^\ell(u_H)[w_H,\delta v_H]\bigr\|_V
\label{eq:numerics-curvature-duality-bound}
\end{equation}
where the second factor is bounded through the tangent regularity constant \(L_{\widetilde M,1}^\ell\) of Section \ref{sec:error-analysis}. For the ideal manifold the residual factor vanishes at the solution. For a localized or learned manifold, only the tangential part of the residual is driven to zero by the outer iteration, and the full fine residual at the computed solution remains of the size of the localization and learning defects, \(\|\mathcal R(\widetilde u_{\mathrm{ms}}^\ell)\|_{V_h'}\le L_A\|e_h\|_V\) by \eqref{eq:error-residual-at-computed-state}. The neglected curvature contribution is therefore small of the order of the reconstruction defect. Third, assembling \(D^2\widetilde M_h^\ell\) would require differentiating the per-patch tangent solves a second time and, in the learned variant of Section \ref{subsec:numerics-learned-correctors}, using second derivatives of the networks. In all experiments below the outer iteration with the Galerkin Jacobian converges in \(2\)--\(4\) steps.

\subsection{Localization and Convergence of the Patch-Solved Method}
\label{subsec:numerics-experiments-without-learning}

Here the per-patch fine-scale corrections are solved by the inner Newton iteration; no surrogate is involved. The fine-scale solution \(u_h\) in \eqref{eq:model-fine-scale-problem}, obtained by full \(P_1\) Newton on the fine mesh, is used as a reference solution. Recall that, by the exactness of the global map, the ideal (no localization) multiscale solution equals \(u_h\).

\paragraph{\boldmath Convergence of the Localized Multiscale Solution in \(\ell\).}
\label{par:numerics-convergence-in-ell}
We solve \eqref{eq:computable-coarse-problem} on a fixed coarse mesh of size \(H=1/8\) with the localization parameter ranging from \(\ell=1,\ldots,6\). The relative energy-norm error is reported in Figure \ref{fig:numerics-l-to-ideal} and Table \ref{tab:numerics-l-to-ideal}, which show exponential decay in \(\ell\). The decay rate is approximately \(1.0\) per layer for both coefficients. The random checkerboard is only marginally harder than the smooth medium, by a constant factor of about \(1.5\) and a slightly smaller rate, despite its piecewise-constant non-periodic structure. These results provide numerical evidence that the localization defect $\eta_{\mathrm{loc}}(H,\ell)$ appearing in Proposition~\ref{prop:patch-solved-localization-bound} decays exponentially in $\ell$. A rigorous a priori proof of this decay is not attempted here; it is deferred to the companion analysis manuscript.

\begin{figure}[htbp]
\centering
\includegraphics[width=0.78\textwidth]{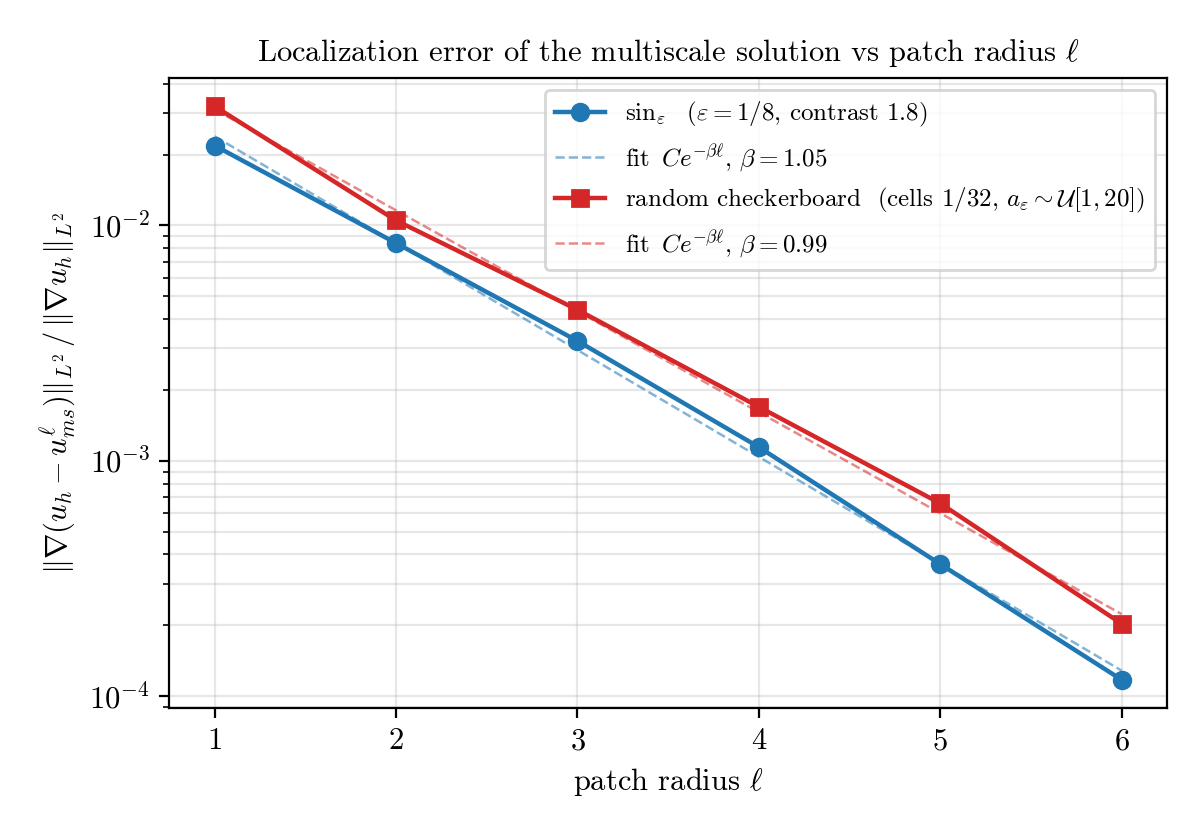}
\caption{Relative \(H^1\) error \(\|\nabla(u_h-u_{\mathrm{ms}}^\ell)\|_{\Omega}/\|\nabla u_h\|_{\Omega}\) versus patch radius \(\ell\), on a semilog scale, for the \(\sin_\varepsilon\) and random-checkerboard coefficients, \(H=1/8\), and \(h=1/64\).}
\label{fig:numerics-l-to-ideal}
\end{figure}

\begin{table}[htbp]
\centering
\begin{tabular}{|r|c|c|}
\hline
\(\ell\) & \(\sin_\varepsilon\) & Random checkerboard \\
\hline
\(1\) & \(2.18\times10^{-2}\) & \(3.21\times10^{-2}\) \\
\(2\) & \(8.41\times10^{-3}\) & \(1.05\times10^{-2}\) \\
\(3\) & \(3.23\times10^{-3}\) & \(4.38\times10^{-3}\) \\
\(4\) & \(1.14\times10^{-3}\) & \(1.69\times10^{-3}\) \\
\(5\) & \(3.64\times10^{-4}\) & \(6.59\times10^{-4}\) \\
\(6\) & \(1.17\times10^{-4}\) & \(2.03\times10^{-4}\) \\
\hline
Fit \(\beta\) & \(\approx1.05\) & \(\approx0.99\) \\
\hline
\end{tabular}
\caption{Relative \(H^1\) error \(\|\nabla(u_h-u_{\mathrm{ms}}^\ell)\|_{\Omega}/\|\nabla u_h\|_{\Omega}\).}
\label{tab:numerics-l-to-ideal}
\end{table}

\paragraph{\boldmath Convergence in the Coarse Mesh Size \(H\).}
\label{par:numerics-convergence-in-h}
We next fix the localization radius and refine the coarse mesh. Keeping the fine reference at \(h=1/64\), we sweep \(H=1/N\in\{1/4,1/8,1/16,1/32\}\), for \(\ell=1,\ldots,4\), with patch-solved fine-scale corrections. The relative energy-norm error against \(u_h\) is reported in Figure \ref{fig:numerics-convergence-h} and Table \ref{tab:numerics-convergence-h}.
We note that as $H$ decreases, the patch radius $\ell$ needs to grow to reduce the error. The dashed envelope marks the graph for $\ell \sim \log(1/H)$. Following this envelope, the error decreases at the linear rate $\mathcal{O}(H)$ in the energy ($H^1$) norm generally expected for such methods; this rate is observed numerically here, while its rigorous justification is deferred to the companion analysis manuscript. The same behaviour holds for both coefficients.


Furthermore, the black curve is plain coarse \(P_1\) FEM for comparison. It sits far above every localized curve and does not converge while \(H\) fails to resolve the fine scale. For the \(\sin_\varepsilon\) medium, the relative FEM error exceeds \(100\%\) for \(H\ge\varepsilon=1/8\) and only begins to fall once \(H<\varepsilon\). By contrast, the localized method reaches about \(1\%\) already at \(H=1/8\) with a few fine-scale correction layers.

\begin{figure}[htbp]
\centering
\includegraphics[width=0.98\textwidth]{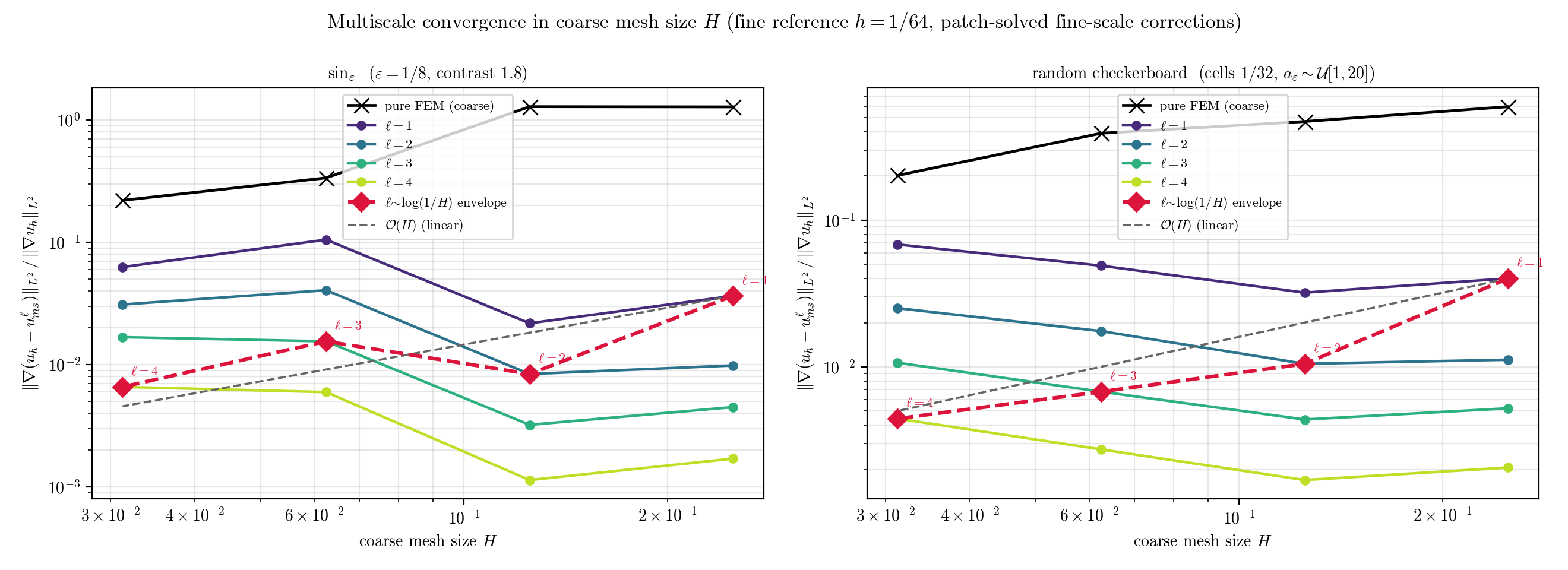}
\caption{Relative \(H^1\) error \(\|\nabla(u_h-u_{\mathrm{ms}}^\ell)\|_{\Omega}/\|\nabla u_h\|_{\Omega}\) versus coarse mesh size \(H\), on a log--log scale, with one curve per localization radius \(\ell\). Left: \(\sin_\varepsilon\). Right: random checkerboard. The black curve is uncorrected coarse \(P_1\) FEM. The dashed envelope grows \(\ell\) by one layer per halving of \(H\), and the gray dashed line is the \(\mathcal O(H)\) reference.}
\label{fig:numerics-convergence-h}
\end{figure}

\begin{table}[htbp]
\centering
\small
\begin{tabular}{|l|c|c|c|c|}
\hline
 & \(H=1/4\) & \(H=1/8\) & \(H=1/16\) & \(H=1/32\) \\
\hline
\multicolumn{5}{|l|}{\(\sin_\varepsilon\)} \\
\hline
Pure FEM & \(1.28\) & \(1.29\) & \(3.37\times10^{-1}\) & \(2.21\times10^{-1}\) \\
\(\ell=1\) & \(3.66\times10^{-2}\) & \(2.18\times10^{-2}\) & \(1.05\times10^{-1}\) & \(6.30\times10^{-2}\) \\
\(\ell=2\) & \(9.88\times10^{-3}\) & \(8.41\times10^{-3}\) & \(4.07\times10^{-2}\) & \(3.11\times10^{-2}\) \\
\(\ell=3\) & \(4.51\times10^{-3}\) & \(3.23\times10^{-3}\) & \(1.56\times10^{-2}\) & \(1.68\times10^{-2}\) \\
\(\ell=4\) & \(1.71\times10^{-3}\) & \(1.14\times10^{-3}\) & \(5.98\times10^{-3}\) & \(6.59\times10^{-3}\) \\
\hline
\multicolumn{5}{|l|}{Random checkerboard} \\
\hline
Pure FEM & \(5.94\times10^{-1}\) & \(4.71\times10^{-1}\) & \(3.91\times10^{-1}\) & \(2.02\times10^{-1}\) \\
\(\ell=1\) & \(4.01\times10^{-2}\) & \(3.21\times10^{-2}\) & \(4.90\times10^{-2}\) & \(6.83\times10^{-2}\) \\
\(\ell=2\) & \(1.12\times10^{-2}\) & \(1.05\times10^{-2}\) & \(1.75\times10^{-2}\) & \(2.52\times10^{-2}\) \\
\(\ell=3\) & \(5.23\times10^{-3}\) & \(4.38\times10^{-3}\) & \(6.78\times10^{-3}\) & \(1.07\times10^{-2}\) \\
\(\ell=4\) & \(2.06\times10^{-3}\) & \(1.69\times10^{-3}\) & \(2.74\times10^{-3}\) & \(4.44\times10^{-3}\) \\
\hline
\end{tabular}
\caption{Relative \(H^1\) error \(\|\nabla(u_h-u_{\mathrm{ms}}^\ell)\|_{\Omega}/\|\nabla u_h\|_{\Omega}\) versus \(H\) and \(\ell\), using vertex patches, patch-solved fine-scale corrections, and the fixed fine reference \(h=1/64\).}
\label{tab:numerics-convergence-h}
\end{table}

\paragraph{Fine-Scale Correction Localization Error on a Patch.}
\label{par:numerics-corrector-localization-error-on-a-patch}
To inspect localization at the level of a single fine-scale correction, fix \(u_H=\Pi_Hu_h\) and consider, on one patch \(i\), the difference between the exact global fine-scale correction \(v_f(u_H)=u_h-u_H\) restricted to the patch and the patch-solved local fine-scale correction \(v_{f,i}^\ell(P_i u_H)\),
\begin{equation}
d_i=v_f(\Pi_Hu_h)\big|_{\omega_i^\ell}-v_{f,i}^\ell(P_i\Pi_Hu_h)
\label{eq:numerics-single-patch-corrector-error}
\end{equation}
The local fine-scale correction is pinned to zero on the artificial patch boundary \(\partial\omega_i^\ell\), while the global one is not. Therefore, \(d_i\) forms a boundary layer at \(\partial\omega_i^\ell\) and decays toward the patch seed, as shown in Figure \ref{fig:numerics-patch-localization}. As \(\ell\) grows, the patch spreads over more coarse elements and the error layer moves outward.

The physically meaningful measure is the error over the active support \(\omega_i^0\), the star of the vertex \(z\), since the partition-of-unity weight \(\varphi_{H,i}\) annihilates the fine-scale correction outside \(\omega_i^0\) before it enters \eqref{eq:localized-reconstruction-map}. The support-restricted seminorm \(\|\nabla d_i\|_{L^2(\omega_i^0)}\) decays exponentially in \(\ell\), 
with rate \(\beta\approx0.6\). By contrast, the seminorm over the full growing patch \(\|\nabla d_i\|_{L^2(\omega_i^\ell)}\) increases with \(\ell\), 
because the integration domain grows. These results motivate the support-restricted contributions in the assembled reconstruction map $M^\ell_h$. 

\begin{figure}[htbp]
\centering
\includegraphics[width=0.95\textwidth]{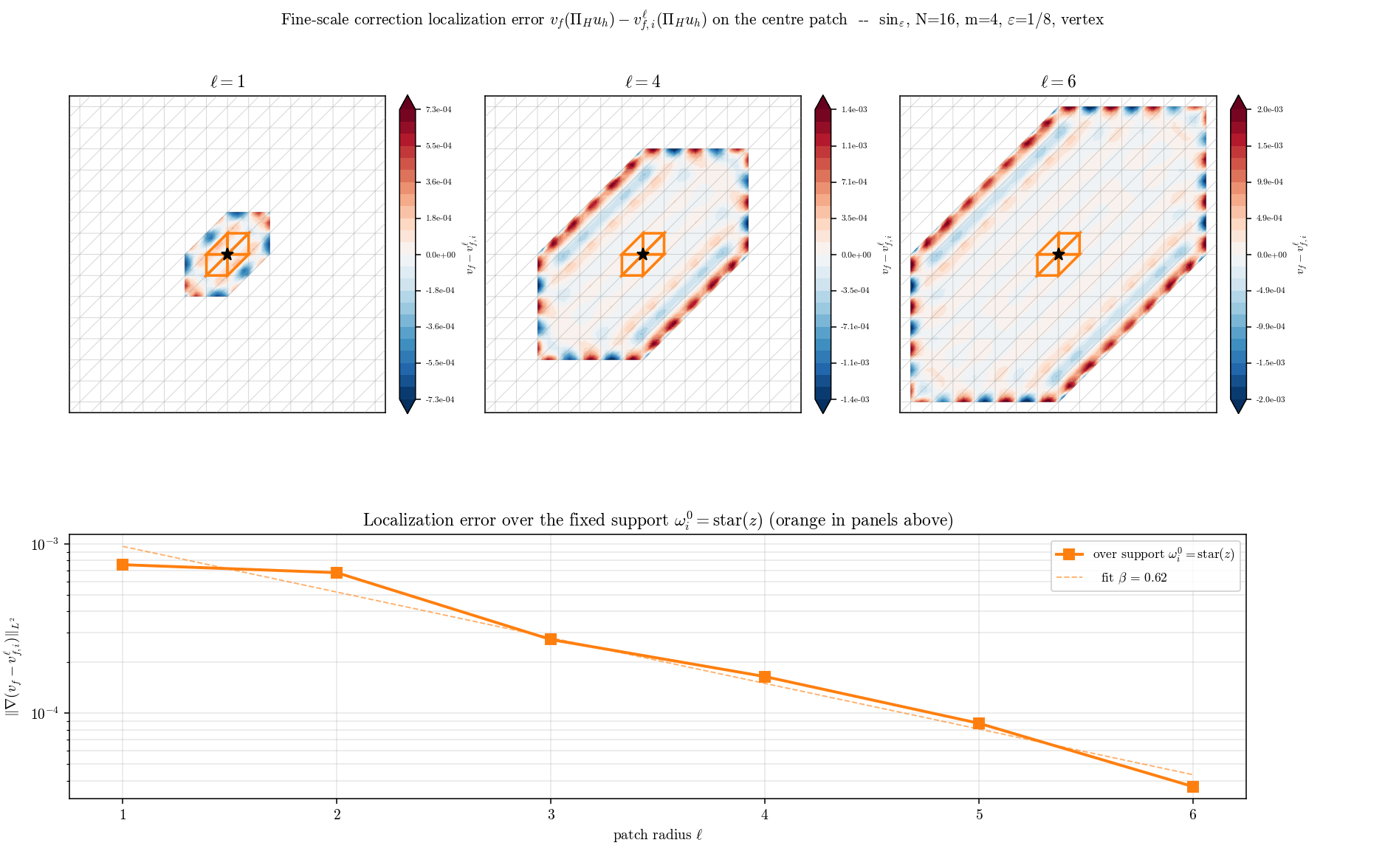}
\caption{Fine-scale correction localization error \(d_i=v_f-v_{f,i}^\ell\) on the center patch for \(\ell=1,4,6\). Top: common spatial frame with the coarse mesh, active support \(\omega_i^0\), and patch seed, each panel colour-scaled to its own range so the boundary layer is visible. Bottom: the support-restricted error \(\|\nabla d_i\|_{L^2(\omega_i^0)}\) versus \(\ell\) with its exponential fit. Setup: \(\sin_\varepsilon\), \(\varepsilon=1/8\), \(N=16\), \(m=4\), \(h=1/64\), and vertex patches.}
\label{fig:numerics-patch-localization}
\end{figure}

\subsection{Network-Interpolated Fine-Scale Corrections}
\label{subsec:numerics-learned-correctors}

We replace each restricted patch-output map by a small neural network \(\mathcal N_{f,i}^{\ell,\theta_i}\), giving the network-interpolated method.

\paragraph{Network and Training Setup.}
\label{par:numerics-network-and-training-setup}
Each patch uses a shallow multilayer perceptron with a single hidden layer and \(\tanh\) activation,
\begin{equation}
\mathcal N_{f,i}^{\ell,\theta_i}(x)=W_2\tanh(W_1\widehat x+b_1)+b_2
\label{eq:numerics-one-hidden-layer-network}
\end{equation}
where \(\widehat x\) is the affinely normalized input. The input \(x=P_i u_H\) is the coarse function restricted to the patch's coarse degrees of freedom. Coarse degrees of freedom on the global Dirichlet boundary $\partial\Omega$ are pinned to zero and carry zero sensitivity. The output is the fine-scale correction on the free (non-Dirichlet) fine nodes of the patch. The hidden width is denoted by \(n_{\mathrm{hid}}\), with \(n_{\mathrm{hid}}=128\) in the production runs. All weights \(\{W_1,b_1,W_2,b_2\}\) are trained. The tangent action required by the outer Newton solve is the analytic one-layer Jacobian,
\begin{equation}
D\mathcal N_{f,i}^{\ell,\theta_i}=\operatorname{diag}(\sigma_y)W_2\operatorname{diag}(1-\tanh^2)W_1\operatorname{diag}(1/\sigma_x)
\label{eq:numerics-network-analytic-jacobian}
\end{equation}
so that, thanks to the shallow single-hidden-layer architecture, no automatic differentiation is needed at solve time.

For each patch, the inputs \(\{u_H^m\}\) are Sobol space-filling points in a box around a centre \(u_H^{\mathrm c}\), perturbing only the active local coarse degrees of freedom. The box half-width is \(0.5|u_H^{\mathrm c}|\), that is, relative radius \(0.5\), per active degree of freedom. Each target \(v_{f,i}^\ell(P_i u_H^m)\) is produced by the inner-Newton patch solve, so generating the dataset is the dominant offline cost. The production runs use \(N_{\mathrm{samp}}=512\) Sobol samples per patch. 

Choosing the box centre \(u_H^{\mathrm c}\) is essential since each network is accurate only near the coarse states it was sampled at, so \(u_H^{\mathrm c}\) must lie where the outer Newton iteration \eqref{eq:computable-coarse-problem} actually evaluates the networks at solve time. We choose \(u_H^{\mathrm c}=u_H^\ell\), the coarse state of the deterministic patch-solved multiscale method, which is precisely the state the learned outer solve converges to. Obtaining \(u_H^\ell\) requires one offline solve of the untrained patch-solved problem, which may seem an unnatural prerequisite for a single solve. Its cost, however, is marginal: the offline stage is dominated by generating the training data, which already runs the deterministic patch solver \(N_{\mathrm{samp}}\) times per patch, so the one extra coarse solve that fixes the centre is negligible. The real justification is amortization: once trained, the networks are reused across many online solves with no further patch solves. This is exploited in Section \ref{subsec:numerics-parabolic-problem}, where the networks are re-used at every step of a time-dependent problem.

The primary loss is the relative \(H^1(\omega_i^\ell)\) energy-seminorm of the output residual in the patch stiffness metric. We add Jacobian, or Sobolev, supervision that matches the network's analytic input Jacobian to the patch-solved fine-scale correction tangent \(D v_{f,i}^\ell\), available from the patch tangent block. This term is measured in the same \(H^1\) patch-energy metric \(A\) as the value loss: writing the Cholesky factorization \(A=LL^\top\), we penalize \(\|L^\top(D\mathcal N_{f,i}^{\ell,\theta_i}-D v_{f,i}^\ell)\|^2\), so that a plain least-squares loss equals the energy-norm error of the tangent mismatch. It uses weight \(\lambda=1\). This directly trains the tangent \(D M_{h,\theta}^\ell\) used by the outer Newton solve. Optimization uses Adam with learning rate \(10^{-3}\), full-batch training for \(2000\) epochs, an \(80/20\) train-validation split, and best-validation-error checkpointing every \(50\) epochs. Training uses JAX/optax; the deployed network, including the forward pass and analytic tangent, uses only \texttt{numpy}/\texttt{scipy}.

\paragraph{Jacobian Supervision and Learnability.}
\label{par:numerics-exploratory-ablations}
The hidden width, sample count, and sampling-box radius were fixed by preliminary single-patch tuning: the out-of-sample error---measured on a fresh Sobol set drawn with an independent seed and disjoint from the training samples---plateaus beyond a moderate width and training-set size, and a tighter box lowers the in-distribution error at the cost of poor out-of-box extrapolation, which motivates the generous radius \(0.5\). On out-of-sample coarse states the learned fine-scale correction reproduces the patch-solved one well, and the \(L^2\)-orthogonality constraint is approximately inherited even though it is not enforced.

We compare value-only (\(\lambda=0\)) and value-plus-Jacobian (\(\lambda=1\)) supervision in the production setting: support-restricted networks at \(h=1/64\), for both coefficients, with the out-of-sample error measured over the support and centred at the multiscale coarse state \(u_H^\ell\) (Table \ref{tab:numerics-learnability}, median over the interior patches). Jacobian supervision lowers both the value error and, more importantly, the tangent error that the outer Newton consumes: by about \(3.4\times\) (value) and \(3.9\times\) (tangent) for the smooth \(\sin_\varepsilon\) medium, and \(5.5\times\) and \(5.9\times\) for the checkerboard. The piecewise-constant fine-scale corrections remain markedly easier to learn than the smooth oscillatory ones at both settings.

\begin{table}[htbp]
\centering
\begin{tabular}{|l|c|c|c|c|}
\hline
 & \multicolumn{2}{c|}{Value Relative \(H^1\)} & \multicolumn{2}{c|}{Tangent Relative Error} \\
\hline
Coefficient & \(\lambda=0\) & \(\lambda=1\) & \(\lambda=0\) & \(\lambda=1\) \\
\hline
\(\sin_\varepsilon\), smooth oscillatory & \(\sim9.1\%\) & \(\sim2.7\%\) & \(\sim24\%\) & \(\sim6.0\%\) \\
Checkerboard, piecewise constant & \(\sim6.2\%\) & \(\sim1.1\%\) & \(\sim16\%\) & \(\sim2.8\%\) \\
\hline
\end{tabular}
\caption{Jacobian supervision in the production setting: support-restricted networks at \(H=1/8\), \(h=1/64\), \(\ell=2\), hidden width \(n_{\mathrm{hid}}=128\), \(N_{\mathrm{samp}}=512\), and box radius \(0.5\), centred at the multiscale coarse state \(u_H^\ell\). Out-of-sample relative errors (median over the interior patches, measured over the support) for value-only training, \(\lambda=0\), and value-plus-Jacobian training, \(\lambda=1\).}
\label{tab:numerics-learnability}
\end{table}

\paragraph{Full Learned Solve.}
\label{par:numerics-full-network-interpolated-solve}
We run the complete network-interpolated method: one network is trained per active patch, and the outer Newton problem \eqref{eq:computable-coarse-problem} is solved with the learned reconstruction \(M_{h,\theta}^\ell\) and its analytic tangent \(D M_{h,\theta}^\ell\). The resulting solution is
\begin{equation}
u_{\mathrm{ms},\theta}^\ell=M_{h,\theta}^\ell(u_{H,\theta}^\ell)
\label{eq:numerics-learned-solution}
\end{equation}
We report the deviation from the patch-solved solution and, for context, the error of both methods against the fine reference. All \(81\) vertex patches are trained with \(\ell=2\), hidden width \(n_{\mathrm{hid}}=128\), \(N_{\mathrm{samp}}=512\), and \(\lambda=1\).

Only the restriction of each fine-scale correction to \(\operatorname{supp}\varphi_{H,i}=\omega_i^0\) enters \eqref{eq:localized-reconstruction-map}, since \(\varphi_{H,i}\) annihilates every node outside \(\omega_i^0\). We therefore shrink each network output layer to the support nodes. At \(h=1/64\), this means about \(136\) active output nodes out of about \(1030\) patch nodes on average. Besides cutting the parameter count of the dominant \(W_2\) block to \(1.70\times10^6\) over all patches, this concentrates capacity on the nodes that survive multiplication by \(\varphi_{H,i}\) and improves accuracy. All results below use these support-restricted networks.


With \(u_H^\ell\)-centred, support-restricted networks, the learned method reproduces the patch-solved multiscale solution to \(0.55\%\) for the \(\sin_\varepsilon\) medium and \(0.12\%\) for the checkerboard. The outer Newton solve with the learned reconstruction converges in \(3\)--\(4\) iterations, and the recovered coarse state is within \(0.02\)--\(0.03\%\) of the patch-solved one. End-to-end against the fine reference, the learned solve barely degrades the patch-solved accuracy, from \(0.84\%\) to \(1.03\%\) for the \(\sin_\varepsilon\) medium and from \(1.05\%\) to \(1.06\%\) for the checkerboard; see Table \ref{tab:numerics-full-solve} and Figure \ref{fig:numerics-full-solve}.

\begin{table}[htbp]
\centering
\small
\begin{tabular}{|l|c|c|c|c|c|}
\hline
Coefficient & L/MS & MS/F & L/F & Coarse & Its \\
\hline
\(\sin_\varepsilon\) & \(5.49\times10^{-3}\) & \(0.84\times10^{-2}\) & \(1.03\times10^{-2}\) & \(2.4\times10^{-4}\) & \(4\) \\
Checkerboard & \(1.23\times10^{-3}\) & \(1.05\times10^{-2}\) & \(1.06\times10^{-2}\) & \(2.8\times10^{-4}\) & \(3\) \\
\hline
\end{tabular}
\caption{Full network-interpolated solve compared with the patch-solved method. The setup uses $H=1/8$, \(\ell=2\), vertex patches, hidden width \(n_{\mathrm{hid}}=128\), \(N_{\mathrm{samp}}=512\), \(\lambda=1\), support-restricted output, \(u_H^\ell\)-centred training, and \(h=1/64\). Errors are relative \(H^1\) errors. The columns L/MS, MS/F, and L/F denote learned versus patch-solved multiscale, multiscale versus fine reference, and learned multiscale versus fine reference.}
\label{tab:numerics-full-solve}
\end{table}

\begin{figure}[htbp]
\centering
\includegraphics[width=0.9\textwidth]{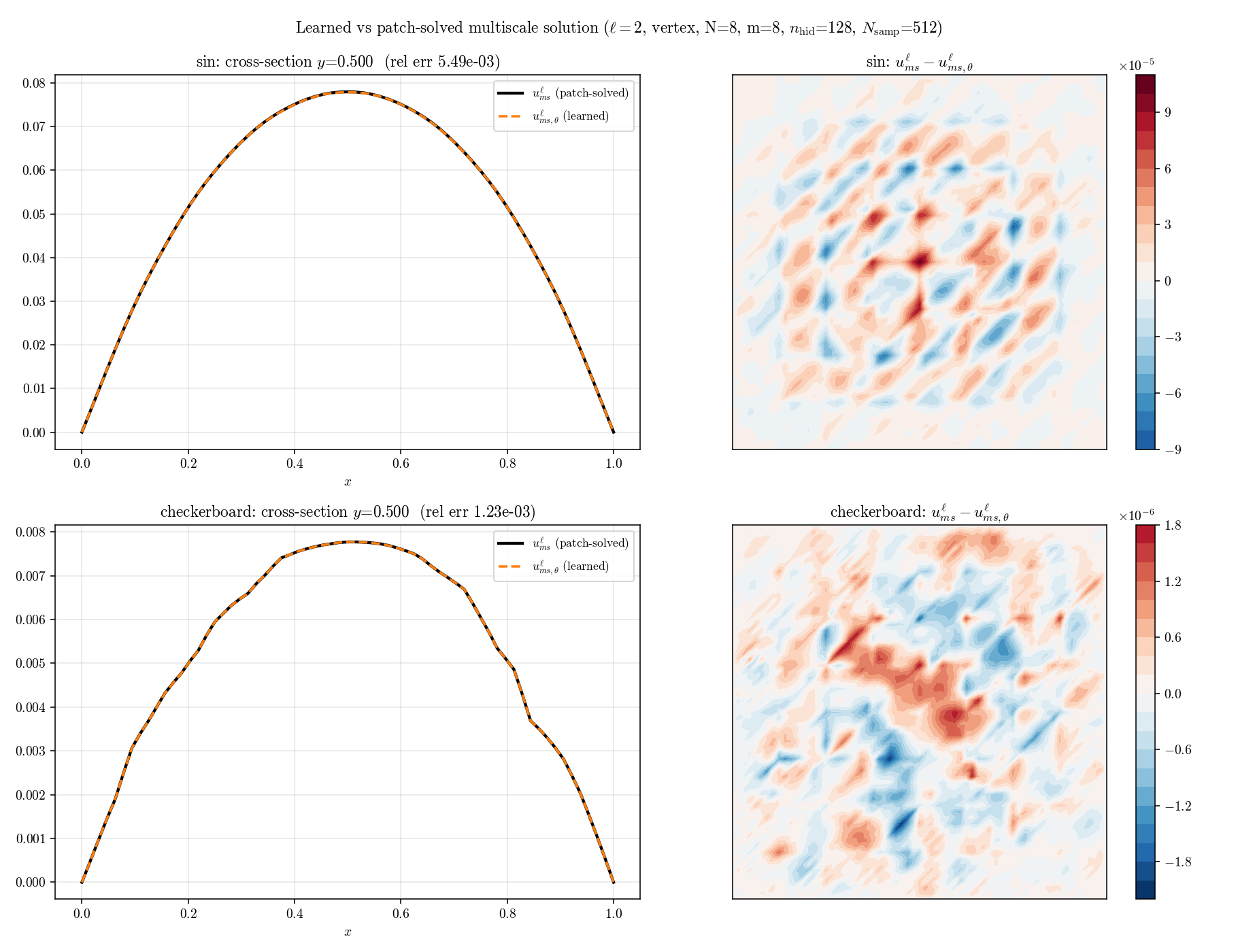}
\caption{Learned versus patch-solved multiscale solution with support-restricted networks. Left: cross-section at \(y=0.5\) of \(u_{\mathrm{ms}}^\ell\) and \(u_{\mathrm{ms},\theta}^\ell\). Right: difference field. Top: \(\sin_\varepsilon\). Bottom: random checkerboard.}
\label{fig:numerics-full-solve}
\end{figure}

\paragraph{Per-Patch Validation of the Restricted Networks.}
\label{par:numerics-patch-validation}
Beyond the assembled solve, we validate the trained restricted networks at the level of individual patches. For each trained patch we compare the learned fine-scale correction to the patch-solved Newton one on a fresh out-of-sample Sobol set centred at the multiscale coarse state \(u_H^\ell\), measuring the relative \(H^1\) error over the support \(\omega_i^0\) --- the only part of each correction that survives the partition-of-unity blend in \eqref{eq:localized-reconstruction-map}. The per-patch out-of-sample median error reaches about \(5\%\) for the \(\sin_\varepsilon\) medium (overall median \(\approx2\%\)) and stays below \(1.6\%\) for the checkerboard (overall median \(\approx0.7\%\)). The heatmaps and seed cross-sections confirm that the network reproduces the exact correction on the support, and the per-patch error bars (Figures \ref{fig:numerics-validate-patch-network} and \ref{fig:numerics-validate-patch-network-checkerboard}) show the piecewise-constant corrections are easier to learn, in agreement with the aggregate accuracy above.

\begin{figure}[htbp]
\centering
\includegraphics[width=0.9\textwidth]{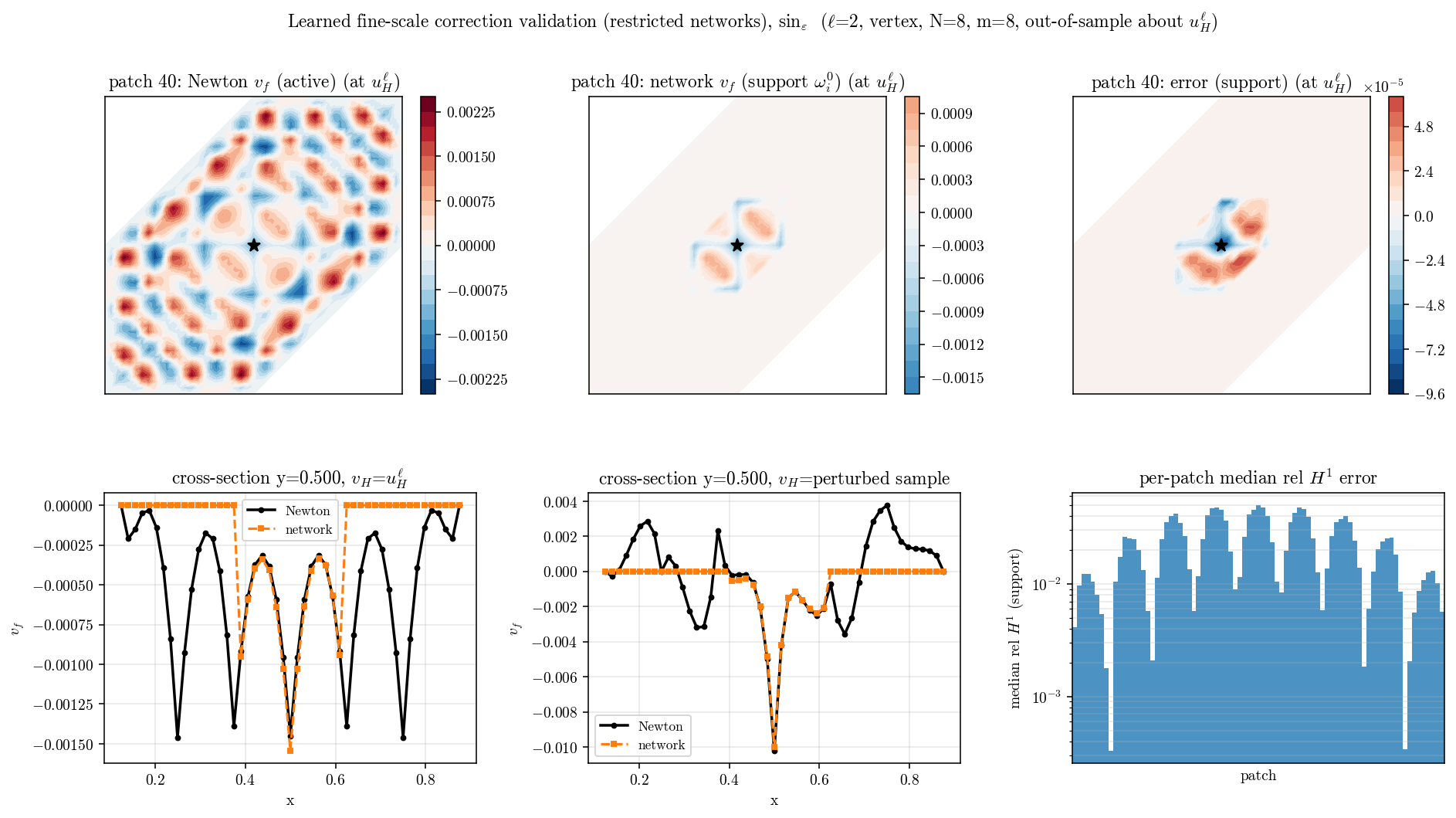}
\caption{Per-patch validation of the support-restricted networks for the \(\sin_\varepsilon\) medium, on an out-of-sample Sobol set centred at the multiscale coarse state \(u_H^\ell\). Top: the patch-solved Newton fine-scale correction on the active patch, the network output on the support \(\omega_i^0\), and their difference. Bottom: cross-sections through the patch seed (at \(u_H^\ell\) and at a perturbed sample) and the per-patch out-of-sample median relative \(H^1\) error over the support.}
\label{fig:numerics-validate-patch-network}
\end{figure}

\begin{figure}[htbp]
\centering
\includegraphics[width=0.9\textwidth]{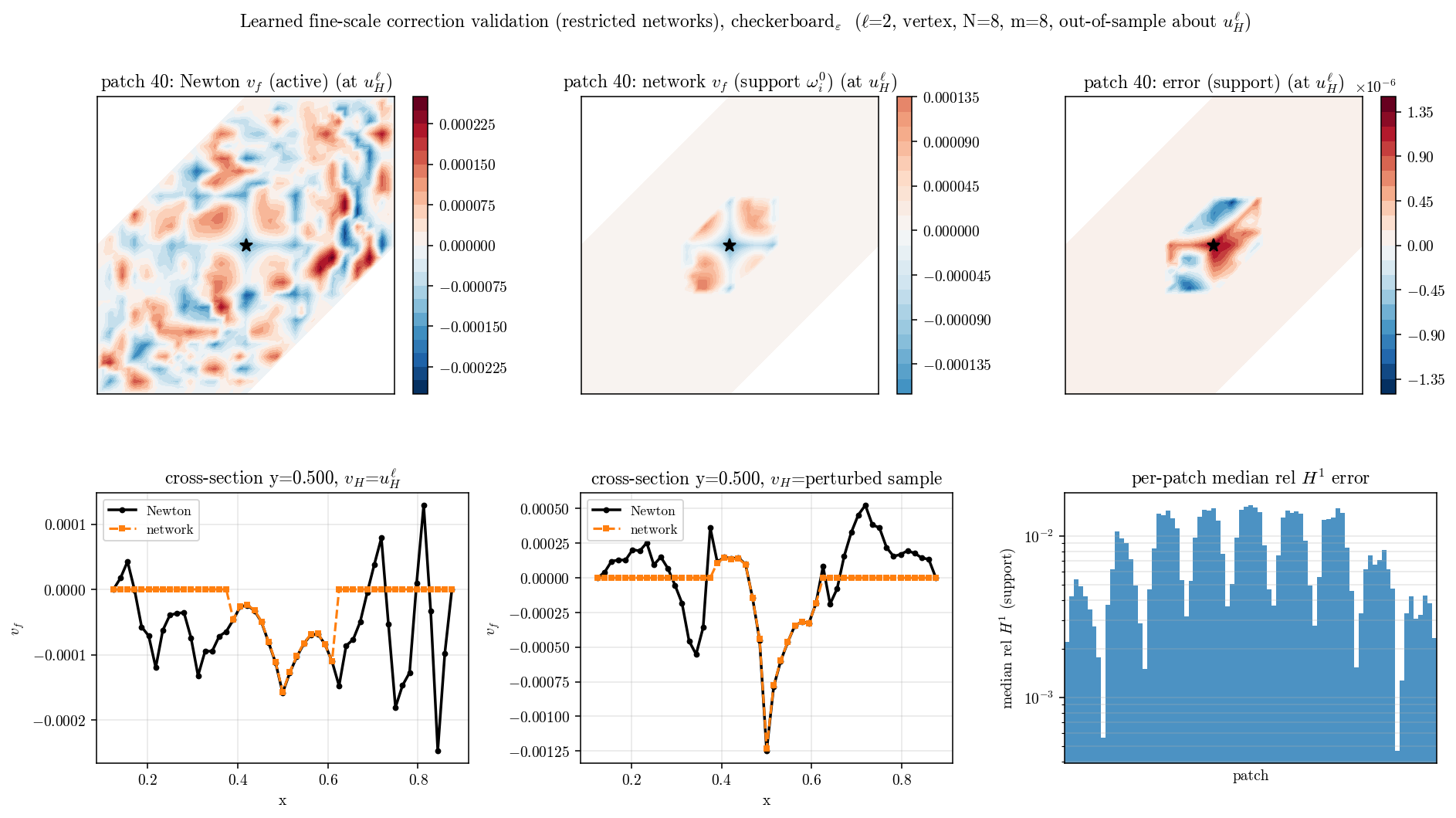}
\caption{Per-patch validation of the support-restricted networks for the random checkerboard, in the same format as Figure \ref{fig:numerics-validate-patch-network}. The per-patch out-of-sample median relative \(H^1\) error over the support stays below \(1.6\%\) (overall median \(\approx0.7\%\)), smaller than for the smooth medium, consistent with the piecewise-constant corrections being easier to learn.}
\label{fig:numerics-validate-patch-network-checkerboard}
\end{figure}

\subsection{Parabolic Problem by Backward Euler}
\label{subsec:numerics-parabolic-problem}

\paragraph{Model Problem.}
\label{par:numerics-parabolic-model-problem}
As a more demanding test, we add a time derivative. On \(\Omega=(0,1)^2\times(0,T]\), we seek \(u(\cdot,t)\) with
\begin{equation}
\partial_t u-\nabla\cdot\bigl(a_\varepsilon(1+\alpha|\nabla u|^2)\nabla u\bigr)=f\quad \text{in }\Omega,\quad u=0\quad \text{on }\partial\Omega,\quad u(\cdot,0)=u_0
\label{eq:numerics-parabolic-problem}
\end{equation}
using the same monotone nonlinear diffusion as in Section \ref{subsec:numerics-model-problem}, with \(f\equiv1\) and \(\alpha=1\). We use either a smooth bump \(u_0=x(x-1)y(y-1)\) or \(u_0\equiv0\).

\paragraph{Fine Reference Scheme.}
\label{par:numerics-fine-reference-scheme}
Backward Euler with step \(\tau\) on the fine space reads: given \(u_h^{n-1}\in V_h\), find \(u_h^n\in V_h\) with
\begin{equation}
\frac{1}{\tau}(u_h^n-u_h^{n-1},v_h)_\Omega + A(u_h^n;v_h) = (f,v_h)_\Omega\quad \forall v_h\in V_h
\label{eq:numerics-backward-euler-fine}
\end{equation}
This problem is solved using a Newton iteration at each step, and the sequence \(\{u_h^n\}\) is the reference solution.

\paragraph{Multiscale Approximation.}
\label{par:numerics-parabolic-multiscale-approximation}
For the multiscale approach, the fine-scale corrections---and hence \(M_h^\ell\) and \(D M_h^\ell\)---are the time-independent stationary ones. The mass term and the history in the parabolic problem enter only the coarse manifold equation. This means that the learned corrections are computed once in an offline stage and can then be reused in every time step, giving the fine-scale corrections by a simple forward pass.

For the multiscale solution each step solves the problem: find \(u_H^n\in V_H\) such that, for all \(\delta v_H\in V_H\),
\begin{align}
&\frac{1}{\tau}\bigl(M_h^\ell(u_H^n)-u_{\mathrm{ms}}^{n-1},D M_h^\ell(u_H^n)[\delta v_H]\bigr)_\Omega \notag\\
&\qquad +A\bigl(M_h^\ell(u_H^n);D M_h^\ell(u_H^n)[\delta v_H]\bigr)
=\bigl(f,D M_h^\ell(u_H^n)[\delta v_H]\bigr)_\Omega
\quad \forall \delta v_H\in V_H
\label{eq:numerics-backward-euler-coarse}
\end{align}
and the reconstructed solution is $u^n_{\mathrm{ms}} = M^\ell_h(u^n_H)$. In matrix form, with \(\mathbf M_h\) the fine mass matrix, \(A_h(\cdot)\) the nonlinear operator action, \(b_h\) the load vector, and history \(u_{\mathrm{ms}}^{n-1}=M_h^\ell(u_H^{n-1})\), the equation is
\begin{equation}
D M_h^\ell(u_H^n)^\top\biggl[\frac{1}{\tau}\mathbf M_h\bigl(M_h^\ell(u_H^n)-u_{\mathrm{ms}}^{n-1}\bigr)+A_h(M_h^\ell(u_H^n))-b_h\biggr]=0
\label{eq:numerics-backward-euler-coarse-matrix-form}
\end{equation}
It is solved by Newton with the Galerkin Jacobian \(J_H(v_H)=(D M_h^\ell(v_H))^\top(\tau^{-1}\mathbf M_h+A'(M_h^\ell(v_H)))D M_h^\ell(v_H)\). Relative to the stationary solve \eqref{eq:computable-coarse-problem}, the only additions are the mass-history term in the residual and \(\tau^{-1}\mathbf M_h\) in the Jacobian. As in the stationary case, the omitted curvature term pairs the full backward-Euler fine residual, now including the mass-history term, with \(D^2M_h^\ell\); it is of the size of the reconstruction defect at convergence, and dropping it does not affect the converged time steps. At each step, the network-multiscale solution \(u_{\mathrm{ms},\theta}^n\) is compared to the patch-solved multiscale solution \(u_{\mathrm{ms}}^n\) and to the fine reference \(u_h^n\).

\paragraph{Results and Distribution Shift.}
\label{par:numerics-parabolic-results-and-distribution-shift}
With \(\tau=0.02\) over \(10\) steps, Figures \ref{fig:numerics-parabolic-sin} and \ref{fig:numerics-parabolic-checkerboard} show that the network-multiscale solution reproduces the patch-solved multiscale solution to its stationary accuracy once the trajectory lies in the networks' training neighborhood. The error is approximately \(0.5\%\) for the \(\sin_\varepsilon\) medium and \(0.12\%\) for the checkerboard. The solution tracks the fine reference to approximately \(1.0\%\) for the \(\sin_\varepsilon\) medium and \(1.06\%\) for the checkerboard, close to the patch-solved multiscale discretization error (\(0.84\%\) and \(1.05\%\), respectively), so the learned interpolation adds little on top of the discretization floor. The outer Newton solve converges to tolerance \(10^{-8}\) at every step, requiring \(1\)--\(3\) iterations after relaxation.

The first-step accuracy is governed by how far the initial coarse state lies from the steady state, which is the networks' training centre. This distance is measured by \(\|u_H^n-u_H^{\mathrm{stat}}\|/\|u_H^{\mathrm{stat}}\|\) in the right panels of Figures \ref{fig:numerics-parabolic-sin} and \ref{fig:numerics-parabolic-checkerboard}. Which initial condition is benign is coefficient-dependent. For the \(\sin_\varepsilon\) medium, the bump starts close to the steady state, at distance \(0.16\), and the network error is uniform at about \(0.5\%\). The zero initial condition starts farther away, at distance \(0.72\), and the network error spikes to about \(6.7\%\) at the first step before recovering. For the checkerboard, the situation reverses: the zero initial state starts close, at distance \(0.21\), while the bump starts far, at distance \(1.38\), and spikes to about \(4.4\%\) before recovering. This is an accuracy effect, not a solver effect; the mass term \(\tau^{-1}\mathbf M_h\) keeps the coarse Jacobian well conditioned.

\begin{figure}[htbp]
\centering
\includegraphics[width=0.95\textwidth]{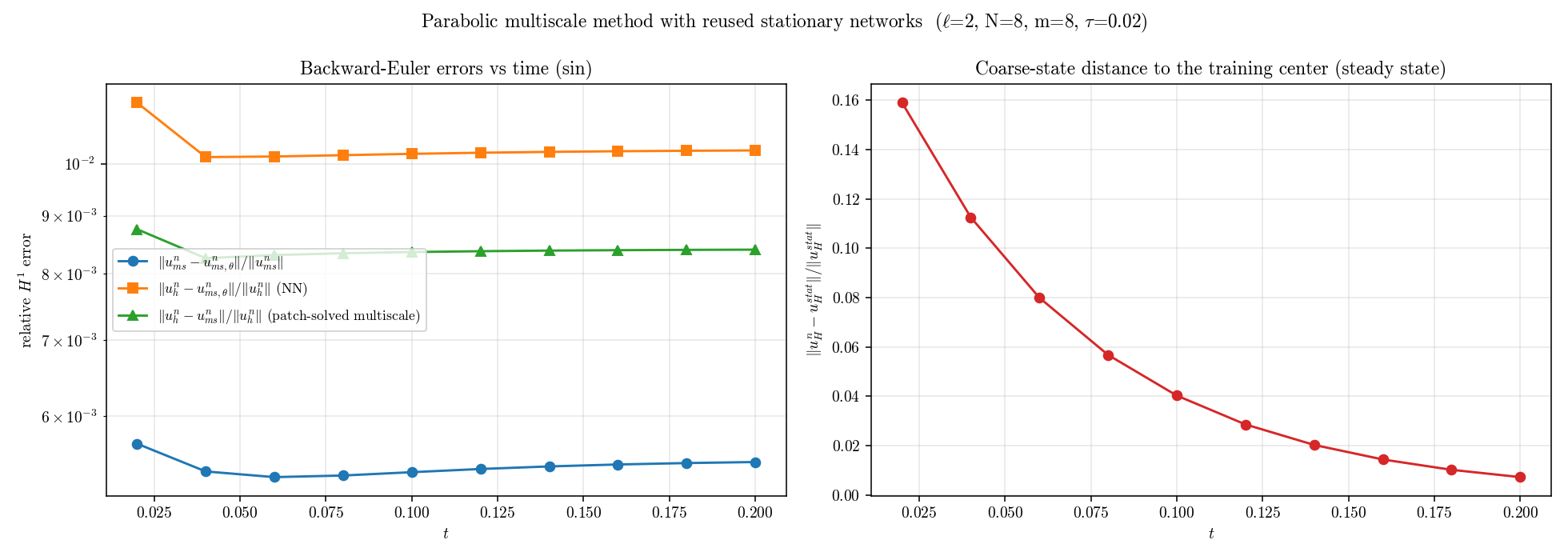}
\includegraphics[width=0.95\textwidth]{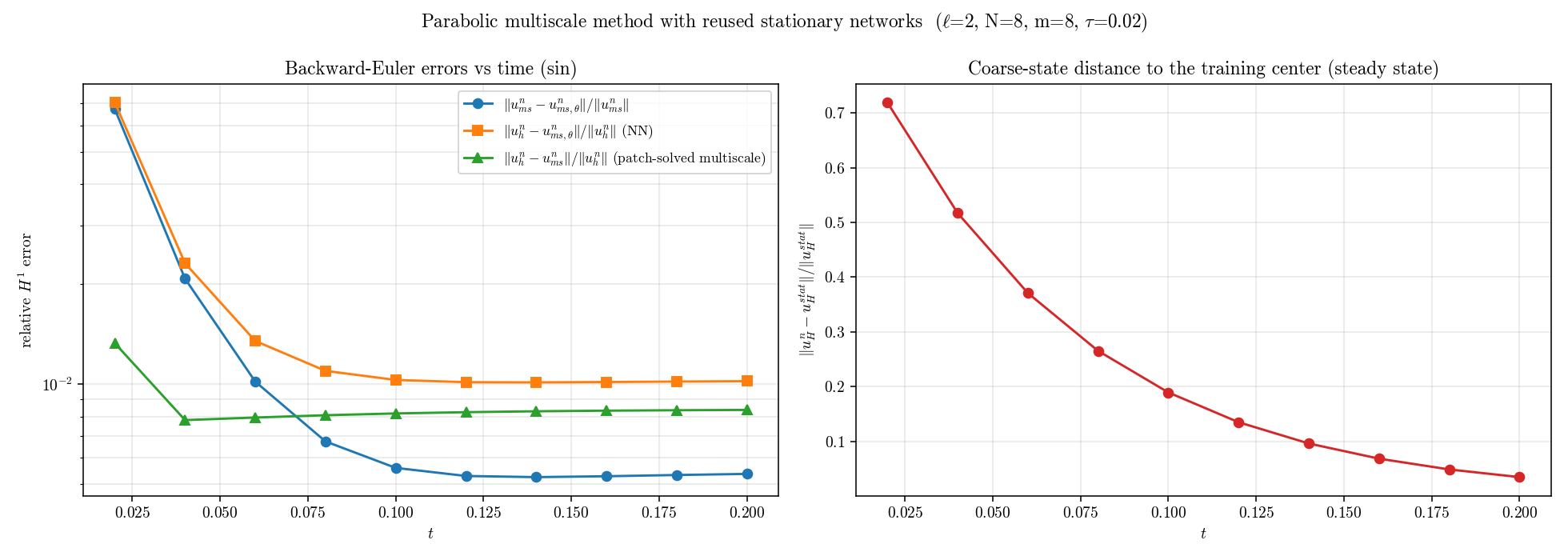}
\caption{\(\sin_\varepsilon\), backward Euler with \(\tau=0.02\) and restricted networks. Top: bump initial data, where the coarse state starts at distance \(0.16\) from steady state and the network error is uniform at about \(0.5\%\). Bottom: zero initial data, where the first step starts at distance \(0.72\) and the network error spikes to about \(6.7\%\) before recovering.}
\label{fig:numerics-parabolic-sin}
\end{figure}

\begin{figure}[htbp]
\centering
\includegraphics[width=0.95\textwidth]{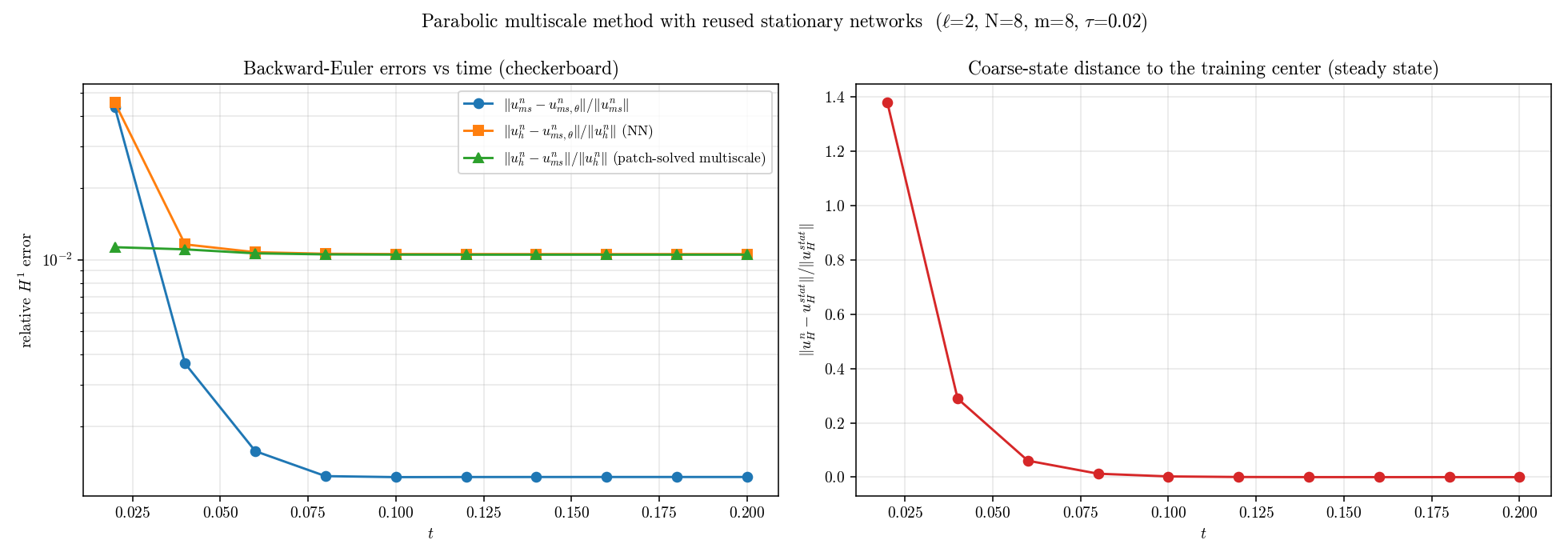}
\includegraphics[width=0.95\textwidth]{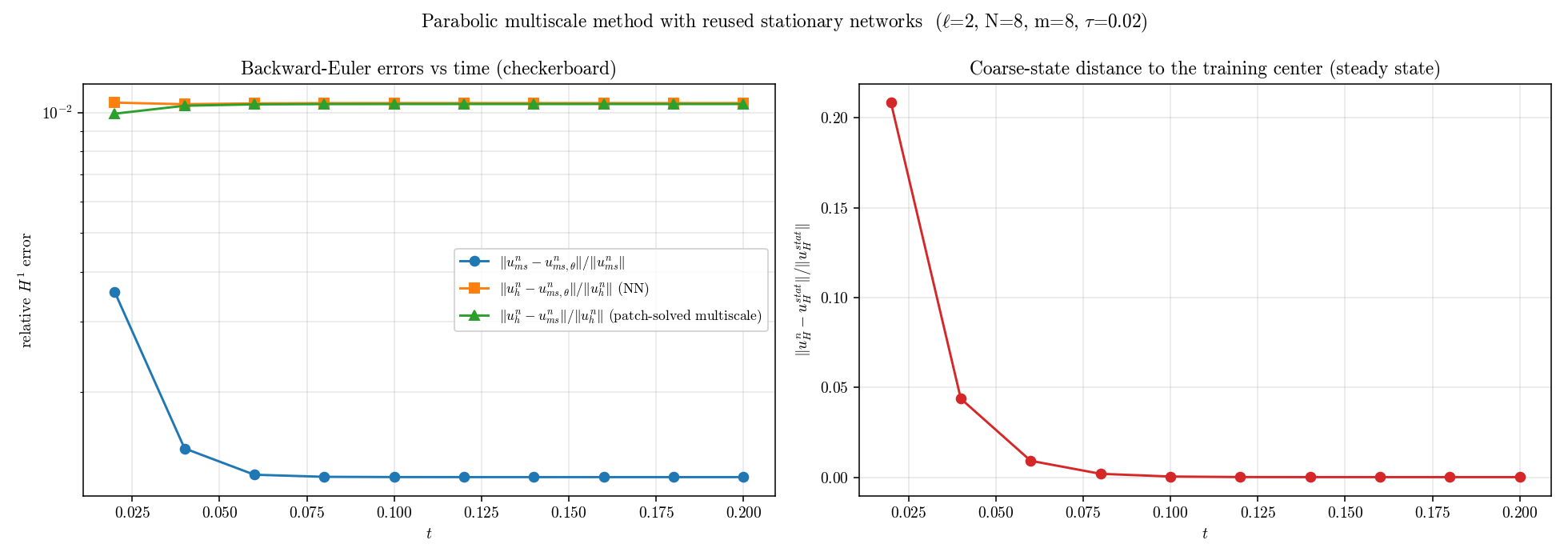}
\caption{Random checkerboard, backward Euler with \(\tau=0.02\) and restricted networks. Top: bump initial data, where the trajectory starts at distance \(1.38\) from steady state and the network error spikes to about \(4.4\%\) before recovering. Bottom: zero initial data, where the trajectory starts at distance \(0.21\) and the network error stays between about \(0.12\%\) and \(0.36\%\).}
\label{fig:numerics-parabolic-checkerboard}
\end{figure}

\paragraph{Potential of Learned Corrections.}
\label{par:numerics-parabolic-potential}
This experiment is small and we do not claim a wall-clock speed-up here. Its purpose is instead to demonstrate the potential of learning the corrections, and to indicate why one would consider learning in the first place. Because the reconstruction is stationary, the same networks are reused unchanged at every time step, so once trained they are amortized over the whole trajectory; the more time steps, the more favourable the offline--online trade-off becomes. Moreover, the offline stage can be made considerably more efficient than the present per-patch training: one may train a single network that is reused on every patch, or on classes of similar patches, rather than one network per patch. One may also replace the data-driven training by an unsupervised approach that minimizes the patch energy or residual directly, thereby requiring no offline data generation. These directions are left for future work.

\paragraph{Implementation Notes.}
\label{par:numerics-implementation-notes}
All experiments share the fixed fine reference mesh \(h=1/64\). The code uses \texttt{numpy}/\texttt{scipy} except for network training, which uses JAX/optax. The runtime network, including the forward pass and analytic tangent, is JAX-free.

%% file: sections/appendix_extended_algorithm_v0.tex

\section{Extended Patch-Solved Algorithm}
\label{sec:extended-patch-solved-algorithm}

This appendix gives the implementation-level patch-solved algorithm used for the stationary numerical experiments. It expands Algorithm \ref{alg:patch-solved-online-solve} by including the cross-mesh constraint, patch caches, and reuse of patch factorizations for tangent assembly. The listing is kept non-floating so that it can break across pages.

\paragraph{Patch-Solved Manifold Newton Method.}
\label{par:appendix-patch-solved-manifold-newton-method}
\noindent\textbf{Algorithm.} Tangent-space multiscale manifold method with outer manifold Newton and inner per-patch Newton.
\begin{algorithmic}[1]
\Statex \textbf{Precompute mesh, operators, and patches}
\State Build nested coarse and fine meshes, and the inclusion \(\iota_H\colon V_H\hookrightarrow V_h\)
\State Assemble the cross-mesh mass matrix \(M_{HF}\) representing the \(V_f=\ker\Pi_H\) constraint
\State Assemble the fine load vector \(b_h\)
\ForAll{patches \(i\in\mathcal I\)}
\State Build the \(\ell\)-ring patch \(\omega_i^\ell\) and the free (non-Dirichlet) node set \(\operatorname{free}_i\)
\State Restrict the constraint matrix to obtain \(B_i\), using the rows of all coarse vertices of \(\overline{\omega_i^\ell}\) interior to \(\Omega\) and the columns of the free fine nodes
\State Store the restricted inclusion, load, and partition-of-unity weights on \(\omega_i^\ell\)
\EndFor
\Statex
\Statex \textbf{Outer Newton iteration on the coarse manifold equation}
\State Set \(u_H\gets0\) on the interior coarse degrees of freedom, with boundary values pinned to zero
\Repeat
\State \(u_h\gets\Call{Reconstruct}{u_H}\)
\Comment{evaluate \(M_h^\ell(u_H)\) by inner patch Newton solves}
\State \(DM\gets\Call{AssembleTangent}{u_H}\)
\Comment{evaluate \(D M_h^\ell(u_H)\), reusing patch factorizations}
\State \(R_h\gets\operatorname{residual}(u_h)-b_h\)
\Comment{fine residual vector for \(A(u_h;\cdot)-L(\cdot)\)}
\State \(r_H\gets(DM^\top R_h)|_{\operatorname{int}}\)
\If{\(\|r_H\|_\infty<\texttt{tol}\)}
\State \Return \(u_H,u_h\)
\EndIf
\State \(A_h'\gets\operatorname{tangent\_matrix}(u_h)\)
\Comment{matrix representation of \(A'(u_h)\)}
\State \(J_H\gets(DM^\top A_h'DM)|_{\operatorname{int}}\)
\Comment{Galerkin Jacobian; the \(D^2M_h^\ell\) term is dropped}
\State Solve \(J_H\delta u_H=-r_H\) on the interior coarse degrees of freedom
\State Update \(u_H\gets u_H+\delta u_H\)
\Until{the outer iteration has converged}
\Statex
\Function{Reconstruct}{\(u_H\)}
\Comment{forward map \(M_h^\ell(u_H)\)}
\State \(u_h\gets\iota_Hu_H\)
\ForAll{patches \(i\in\mathcal I\)}
\State \(v_{f,i}\gets\Call{PatchEvaluate}{i,u_H}\)
\Comment{inner Newton solve for the local problem \eqref{eq:patch-correction-problem}}
\State Scatter-add \(\varphi_{H,i}v_{f,i}\) to the fine vector on \(\omega_i^\ell\)
\EndFor
\State Project the blended correction with \(I-\Pi_H\), as in \eqref{eq:localized-reconstruction-map}
\State \Return \(u_h\)
\EndFunction
\Statex
\Function{PatchEvaluate}{\(i,u_H\)}
\Comment{inner Newton solve on \(\omega_i^\ell\)}
\State \(g_i\gets\iota_Hu_H|_{\omega_i^\ell}\)
\State Initialize \(v\) from the patch cache or set \(v\gets0\)
\Repeat
\State \(u\gets g_i+v\)
\State \(r\gets(\operatorname{residual}_{\omega_i^\ell}(u)-b_i)|_{\operatorname{free}_i}\)
\If{\(\|r\|/\|r_0\|<\texttt{tol}\)}
\State \textbf{break}
\EndIf
\State \(K\gets\operatorname{tangent\_matrix}_{\omega_i^\ell}(u)\)
\State Solve the constrained patch system with block matrix \(\bigl[\begin{smallmatrix}K_{aa}&B_i^\top\\ B_i&0\end{smallmatrix}\bigr]\)
\State Update the free correction degrees of freedom \(v[\operatorname{free}_i]\gets v[\operatorname{free}_i]+dq\)
\Until{the local Newton iteration has converged}
\State \(K^\ast\gets\operatorname{tangent\_matrix}_{\omega_i^\ell}(g_i+v)\)
\State Factorize the final constrained patch matrix \(S_i\gets\operatorname{Factor}(K^\ast,B_i)\)
\State Cache \((v,K^\ast,S_i)\) for the current coarse state \(u_H\)
\State \Return \(v\)
\EndFunction
\Statex
\Function{AssembleTangent}{\(u_H\)}
\Comment{assemble \(D M_h^\ell(u_H)\)}
\State \(DM\gets\iota_H\)
\ForAll{patches \(i\in\mathcal I\)}
\State Retrieve \((v,K^\ast,S_i)\) from \textsc{PatchEvaluate}
\State Let \(Z_i\) be the coarse basis functions whose support intersects \(\omega_i^\ell\)
\State Solve the linearized patch problem \eqref{eq:localized-tangent-patch-problem} for all directions in \(Z_i\) using the cached factorization \(S_i\)
\State Scatter-add the partition-of-unity weighted tangent blocks to \(DM\)
\EndFor
\State Project the blended tangent correction with \(I-\Pi_H\), as in \eqref{eq:localized-tangent-reconstruction}
\State \Return \(DM\)
\EndFunction
\end{algorithmic}

%% file: sections/acknowledgements_v1.tex

\paragraph{Acknowledgements.}
\label{par:acknowledgements}
Mats G. Larson was supported in part by the Swedish Research Council, Grants Nos. 2021-04925 and 2025-05562, the Knut and Alice Wallenberg Foundation, Grant No. KAW 2025.0277, and the Swedish Research Programme Essence. Anna Persson acknowledges support from the Swedish Research Council Grant no. 2022-03543. 

\paragraph{Use of generative AI.}
Generative AI tools, including OpenAI's ChatGPT and Anthropic's Claude, were used to assist with language editing, \LaTeX{} formatting, figure-caption drafting, reference suggestions, and code generation. All AI-assisted material was critically reviewed, verified, and edited by the authors. The authors take full responsibility for the mathematical content, numerical results, and conclusions of the manuscript.

%% file: sections/references_v4.tex

%% file: sections/author_addresses_v1.tex

\bigskip
\bigskip
\noindent
\footnotesize \textbf{Authors' Addresses:}

\smallskip
\noindent
Mats G. Larson, \quad \hfill Department of Mathematics and Mathematical Statistics, Ume{\aa} University, Sweden\\
\texttt{mats.larson@umu.se}

\smallskip
\noindent
Anna Persson, \quad \hfill Department of Information Technology, Uppsala University, Sweden\\
\texttt{apersson@it.uu.se}